\documentclass[12pt]{amsart}
\usepackage{latexsym}
\usepackage{amssymb} 
\usepackage{mathrsfs}
\usepackage{amsmath}
\usepackage{latexsym}
\usepackage{delarray}
\usepackage{amssymb,amsmath,amsfonts,amsthm,
mathrsfs}

\numberwithin{equation}{section}
\theoremstyle{plain}
\newtheorem{thm}{Theorem}[section]
\newtheorem{lem}[thm]{Lemma}
\newtheorem{proposition}[thm]{Proposition}

\newtheorem{remark}[thm]{Remark}

\newcommand{\irm}{{\rm i}}

\newcommand{\wt}[1]{\widetilde{#1}}

\renewcommand{\S}{\mathscr{S}}

\newcommand{\mb}[1]{\ensuremath{\mathbb{#1}}}
\newcommand{\N}{\mb{N}}

\newcommand{\R}{\mb{R}}
\newcommand{\C}{\mb{C}}

\newcommand{\lara}[1]{\langle #1 \rangle}

\newfont{\bigmath}{cmr12 at 13pt}

\newfont{\grecomath}{cmmi12 at 15pt}

\newcommand{\esp}{\mathrm{e}}

\newcommand{\beq}{\begin{equation}}
\newcommand{\eeq}{\end{equation}}

\newcommand{\eps}{\varepsilon}

\renewcommand{\Im}{\ensuremath{\mathrm{Im}}}

\title[Well-posedness of non-homogeneous weakly hyperbolic equations]{
Weakly hyperbolic equations with non-analytic 
coefficients and lower order terms
}

\author[Claudia Garetto]{Claudia Garetto}
\address{
  Claudia Garetto:
  \endgraf
  Department of Mathematics
  \endgraf
  Imperial College London
  \endgraf
  180 Queen's Gate, London SW7 2AZ
  \endgraf
  United Kingdom
  \endgraf
  {\it E-mail address} {\rm c.garetto@imperial.ac.uk}
  }
\author[Michael Ruzhansky]{Michael Ruzhansky}
\address{
  Michael Ruzhansky:
  \endgraf
  Department of Mathematics
  \endgraf
  Imperial College London
  \endgraf
  180 Queen's Gate, London SW7 2AZ
  \endgraf
  United Kingdom
  \endgraf
  {\it E-mail address} {\rm m.ruzhansky@imperial.ac.uk}
  }

\thanks{
The first author was supported by
the Imperial College Junior Research Fellowship.
The second author was supported by the
EPSRC Leadership Fellowship EP/G007233/1.
}
\date{}


\begin{document}

\maketitle

\begin{abstract}
In this paper we consider weakly hyperbolic equations of higher orders
in arbitrary dimensions with time-dependent coefficients and lower order terms.
We prove the
Gevrey well-posedness of the Cauchy problem
under $C^k$-regularity of coefficients of the principal
part and natural Levi conditions on lower order
terms which may be only continuous.
In the case of analytic coefficients in the principal part
we establish the $C^\infty$ well-posedness. The proofs are based on using
the quasi-symmetriser for the corresponding system. The main novelty
compared to the existing literature is the possibility
to include lower order terms to the equation
(which have been untreatable until now in these problems)
as well as considering
any space dimensions. We also give results on the
ultradistributional and distributional well-posedness of the problem,
and we look at new effects for equations with discontinuous
lower order terms.
\end{abstract}

\section{Introduction}
In this paper we study the well-posedness of the weakly hyperbolic Cauchy problem
\beq
\label{CP}
\left\{
\begin{array}{cc}
D^m_t u+\sum_{j=0}^{m-1} A_{m-j}(t,D_x)D_t^j u=0,&\quad (t,x)\in[0,T]\times\R^n,\\
D^{l-1}_t u(0,x)=g_{l}(x),&\quad l=1,...,m,
\end{array}
\right.
\eeq
where each $A_{m-j}(t,D_x)$ is a differential operator of order $m-j$ with continuous coefficients  depending only on $t$. Later we will also relax the continuity
assumption replacing it by the boundedness.
As usual, $D_t=\frac{1}{\irm}\partial_t$ and $D_x=\frac{1}{\irm}\partial_x$. Let $A_{(m-j)}$ denote the principal part of the operator $A_{m-j}$ and let $\lambda_l(t,\xi)$, $l=1,...,m$, be the real-valued roots of the characteristic polynomial which we write as
\[
\tau^m+\sum_{j=0}^{m-1}A_{(m-j)}(t,\xi)\tau^j
=\tau^m+\sum_{j=0}^{m-1}\sum_{|\gamma|=m-j}a_{m-j,\gamma}(t)\xi^\gamma\tau^j.
\]
This means that
\[
\tau^m+\sum_{j=0}^{m-1}\sum_{|\gamma|=m-j}a_{m-j,\gamma}(t)\xi^\gamma\tau^j=\prod_{l=1}^m (\tau-\lambda_l(t,\xi)).
\]

The well-posedness of the weakly hyperbolic equations has been a challenging problem for a long time.
For example, even for the second order Cauchy problem in one space dimension,
\begin{equation}\label{EQ:we}
\partial_t^2 u - a(t,x)\partial_x^2 u=0, \; u(0,x)=g_1(x),\; \partial_t u(0,x)=g_2(x),
\end{equation}
up until now there is no characterisation of smooth functions $a(t,x)\geq 0$ for which
\eqref{EQ:we} would be $C^\infty$ well-posed. On one hand, there are sufficient
conditions. For example, Oleinik has shown in \cite{O}  that \eqref{EQ:we} is $C^\infty$
well-posed provided there is a constant $C>0$ such that
$C a(t,x)+\partial_t a(t,x)\geq 0$.
In the case of $a(t,x)=a(t)$ depending only on $t$, when the problem becomes
\begin{equation}\label{EQ:we1}
\partial_t^2 u - a(t)\partial_x^2 u=0, \; u(0,x)=g_1(x),\; \partial_t u(0,x)=g_2(x),
\end{equation}
the Oleinik's
condition is satisfied for $a(t)\geq 0$ with $a'(t)\geq 0$. On the other hand, in the
celebrated paper \cite{CS}, Colombini and Spagnolo constucted a $C^\infty$ function
$a(t)\geq 0$ such that \eqref{EQ:we1} is not $C^\infty$ well-posed.
The situation becomes even more complicated if one adds mixed terms to \eqref{EQ:we1},
even depending only on $t$ and analytic.
For example, the Cauchy problem for the equation
$$
\partial_t^2 u - 2t \partial_t\partial_x u+ t^2\partial_x^2 u=0
$$
is Gevrey $G^s$ well-posed for $s<2$ while it is ill-posed for any $s>2$.
For other positive and negative results for second order equations with
time-dependent coefficients we refer to seminal papers
of Colombini, De Giorgi and Spagnolo \cite{CDS} and
Colombini, Jannelli and Spagnolo \cite{CJS2}, and to
Nishitani \cite{N} for the necessary and sufficient conditions for
the $C^{\infty}$ well-posedness of \eqref{EQ:we} with analytic
$a(t,x)\geq 0$ in one dimension.

A reasonable substitute for the $C^\infty$ well-posedness in the weakly hyperbolic setting
is the well-posedness in the space $G^\infty=\bigcup_{s>1} G^s$. Thus,
Colombini, Jannelli and Spagnolo proved in \cite{CJS} that for every $C^\infty$ function
$a(t)\geq 0$, the Cauchy problem \eqref{EQ:we1} is $G^\infty$ well-posed.
More precisely, they showed that if $a(t)$ is in $C^k$, it is well posed in
$G^s$ with $s\leq 1+k/2$, and if $a(t)$ is analytic, it is $C^\infty$ well-posed.

From another direction, there are also general results for \eqref{CP}. For example,
it was shown by Bronshtein in \cite{B} that, in paticular, the Cauchy problem
\eqref{CP} with $C^\infty$ coefficients is $G^s$ well-posed provided that
$1\leq s<1+\frac{1}{m-1}.$ In some cases, this can be improved. For example,
for constant multiplicities,
see paper \cite{ColKi:02} by Colombini and Kinoshita in one-dimension
(see also D'Ancona and Kinoshita \cite{DK}),
and the authors' paper \cite{GR:11} for further improvements of
Gevrey indices and all dimensions,
with a survey of literature therein.

In this paper our interest in analysing the Cauchy problem \eqref{CP} is motivated by
\begin{itemize}
\item[(A)] allowing any space dimension $n\geq 1$;
\item[(B)] considering the effect of lower order terms or, rather, the properties of the lower
order terms which do not influence the results on the Gevrey well-posedness
(we will look at new effects for both continuous and discontinuous
lower order terms);  the inclusion of lower order terms in this setting has
been untreatable by previous methods;
\item[(C)] providing well-posedness results in spaces of distributions and ultradistributions.
\end{itemize}

Our main reference here is the paper
\cite{KS} of Kinoshita and Spagnolo who
have studied the Cauchy problem \eqref{CP} for operators
with homogeneous symbols
in one dimension, $x\in\R$. Under the condition
\begin{multline}
\label{LC}
\exists M>0:\\[0.3cm]
\lambda_i(t,\xi)^2+\lambda_j(t,\xi)^2\le M(\lambda_i(t,\xi)-\lambda_j(t,\xi))^2,\qquad 1\le i,j\le m, t\in[0,T],  \textrm{ for all }\xi,
\end{multline}
on the roots $\lambda_j(t,\xi)$ they have obtained the following well-posedness result:
\begin{thm}[\cite{KS}]
\label{teo_KS}
Assume that $n=1$ and that the differential operator is homogeneous, i.e.
$A_{m-j}(t,\xi)=A_{(m-j)}(t,\xi)=a_{m-j}(t)\xi^{m-j}$ for all $j=0,\ldots,m-1$.
If $a_{m-j}\in C^\infty([0,T])$ and the characteristic roots are real and satisfy \eqref{LC}, then the Cauchy problem \eqref{CP} is well-posed in any Gevrey space. More precisely, if $a_j\in{C}^k([0,T])$ for some $k\ge 2$ then we have $G^s$-well-posedness for
\[
1\le s<1+\frac{k}{2(m-1)}.
\]
\end{thm}
The proof is based on the construction of a quasi-symmetriser $Q_\eps^{(m)}$ which thanks to the
condition \eqref{LC} is nearly diagonal. Previously, equations of second and third order
with analytic coefficients, still with $n=1$ and without low order terms,
have been analysed by Colombini and Orr\'{u} \cite{CO}. They have shown the $C^\infty$
well-posedness of \eqref{CP} under assumption \eqref{LC}. Moreover, if all the
coefficients $a_{m-j}(t)$ vanish at $t=0$, they showed that the {\em condition
\eqref{LC} is also necessary}.
So, for us it will be natural to adopt \eqref{LC} for our analysis.

Let us briefly discuss the difficulties of aims (A)--(C) above. For the dimensional
extension (A), even under condition \eqref{LC} for the
characteristic roots, for space dependent coefficients
such an extension is impossible, see e.g. Bernardi and Bove \cite{BB},
for examples of second order operators with polynomial coefficients for
which the $C^{\infty}$ well-posedness fails for any $n\geq 2$. It is interesting
to note that for these examples the usual Ivrii--Petkov conditions on lower
order terms are also satisfied. As we will show, the $C^{\infty}$ (and other)
well-posedness holds in our case in any dimension $n\geq 1$ 
since the coefficients depend only on time.
In part (B), the proof of the well-posedness for equations with lower order terms
highlights several interesting and somewhat surprising
phenomena. For example, if the coefficients of
the principal part are analytic and the lower order terms are only bounded
(in particular, they may be discontinuous, or may exhibit more irregular
oscillating behaviour), but the Cauchy data is Gevrey,
we still obtain the solution in Gevrey spaces. Indeed, the Levi conditions in
this paper control the zeros of the lower order terms but not their
regularity. 
Finally, aim (C) is motivated by an interesting and challenging
problem for weakly hyperbolic equations: analysing the propagation of
singularities. For this, in order to be able to use also non-Gevrey
techniques, we need to have first well-posedness in some bigger space.
This will be achieved for the Cauchy problem \eqref{CP}
in the spaces of Beurling Gevrey ultradistributions. A subtle point of this construction
is that we will have to use the Beurling Gevrey ultradistributions and not the usual
Roumieu Gevrey ultradistributional class. In the case of the analytic principal
part we will obtain well-posedness in the usual space of distributions.

In particular, in 
this paper we extend Theorem \ref{teo_KS} to weakly hyperbolic
equations with non-homogeneous symbols and in any space dimension $n\geq 1$,
and find suitable assumptions on the lower order terms for the Gevrey
well-posedness. Already from the beginning we deviate from \cite{KS} by
using pseudo-differential techniques to reduce the equation to the system. This
will allow us to treat all the dimensions $n\geq 1$. However, the main challenge in the present
paper is the analysis of the lower order terms. In fact, in most (if not all) the literature on
the application of the quasi-symmetriser to weakly hyperbolic equations the considered
equations are always assumed to have homogeneous symbols.
It is our intension to show that
the quasi-symmetriser can be effectively used to control parts of the
energy corresponding to the
lower order terms. It is interesting to see the appearing Levi conditions expressing
the dependence of the lower order terms on the principal part of the operator.
Such control becomes possible by exploiting the Sylvester form of the
system corresponding to equation \eqref{CP}, and the structure of the
quasi-symmetriser.

An interesting effect that we observe
is that the results remain true assuming just the
continuity of the lower order terms in time. For example, we will have
the $C^{\infty}$ well-posedness for equations with analytic coefficients
in the principal part and only continuous lower order terms.
Moreover, we give a variant of our results with only assuming the
boundedness of lower order terms in time (instead of continuity).

In this paper we formulate the conditions on the lower order terms in terms of the symbols
$A_{m-j+1}$. Note that in \eqref{CP}, the operator
$A_{m-j+1}(t,D_x)$ is the coefficient in front of the derivative
$D_t^{j-1}$. We assume that there is some constant $C>0$ such that we have
\begin{multline}\label{EQ:lot}
\left|(A_{m-j+1}-A_{(m-j+1)})(t,\xi)\right| \leq
C\sum_{i=1}^m \left|
\sum_{\substack{1\leq \ell_1<\cdots<\ell_{m-j}\leq m \\ \ell_h\not=i\; \forall h}}
\lambda_{\ell_1}(t,\xi)\cdots \lambda_{\ell_{m-j}}(t,\xi)
\right|,
\end{multline}
for all $t\in [0,T]$, $j=1,\ldots, m$ and for $\xi$ away from $0$ (i.e., for $|\xi|\ge R$ for some $R>0$).

For $j=m$, this is the condition on the low order terms coming from the coefficient
in front of $D_t^{m-1}$ , in which case $A_1-A_{(1)}$ is independent of $\xi$, and
assumption \eqref{EQ:lot} should read as
$$
 |(A_1-A_{(1)})(t,\xi)|\leq C, \; t\in [0,T],
$$
which will be automatically satisfied due to the boundedness of
$A_{1}$ in $t$.
In Section \ref{SEC:ex} we will give examples of the condition \eqref{EQ:lot}.
We will also show in treating the case $m=3$
that from the point of view of the desired energy inequality for
\eqref{CP} the assumption \eqref{EQ:lot} is rather natural.

We are now ready to formulate the well-posedness results. Part (i) of
Theorem \ref{theo_GRi} is
the extension of Theorem 1 in \cite{KS} to any space-dimension
and to equations with low order terms.
In the sequel $\mathcal{D}'_{(s)}(\R^n)$ ($\mathcal{E}'_{(s)}(\R^n)$)
denotes the space of Gevrey Beurling (compactly supported) ultradistributions.
For the relevant details on these spaces of ultradistributions and their characterisations,
with their appearance in the analysis of weakly hyperbolic equations, we refer to
our paper \cite{GR:11}, where these have been applied to the low (H\"older)
regularity  constant multiplicities case.
\begin{thm}
\label{theo_GRi} Let $n\geq 1$.
If the coefficients satisfy $A_j(\cdot,\xi)\in{C}([0,T])$ and
$A_{(j)}(\cdot,\xi)\in C^\infty([0,T])$ for all $\xi$,
the characteristic roots are real and satisfy
\eqref{LC}, and the low order terms satisfy \eqref{EQ:lot}, then the Cauchy problem
\eqref{CP} is well-posed in any Gevrey space. More precisely,
for $A_j(\cdot,\xi)\in{C}([0,T])$, we have:
\begin{itemize}
\item[(i)] if $A_{(j)}(\cdot,\xi)\in{C}^k([0,T])$ for some $k\ge 2$ and
$g_j\in G^s(\R^n)$ for $j=1,...,m,$
then there exists a unique solution $u\in C^m([0,T];G^s(\R^n))$ provided that
\[
1\le s<1+\frac{k}{2(m-1)};
\]
\item[(ii)] if $A_{(j)}(\cdot,\xi)\in{C}^k([0,T])$ for some $k\ge 2$ and
$g_j\in \mathcal{E}'_{(s)}(\R^n)$ for $j=1,...,m,$ then there exists a unique
solution $u\in C^m([0,T];\mathcal{D}'_{(s)}(\R^n))$ provided that
\[
1\le s\le 1+\frac{k}{2(m-1)}.
\]
\end{itemize}
\end{thm}
In the case of analytic coefficients, we have $C^\infty$ and distributional well-posedness.
\begin{thm}
\label{theo_GR2i}
If  $A_j(\cdot,\xi)\in{C}([0,T])$ and
the coefficients $A_{(j)}(\cdot,\xi)$ are analytic on $[0,T]$ for all $\xi$, the characteristic roots are real and satisfy \eqref{LC}, and the lower order terms fulfil the conditions \eqref{EQ:lot} then the Cauchy problem \eqref{CP} is $C^\infty$ and distributionally well-posed.
\end{thm}

By $W^{\infty,m}$ we denote the Sobolev space of functions having
$m$ derivatives in $L^{\infty}$.
In the case of discontinuous but bounded lower order terms we have the
following:
\begin{thm}
\label{theo_GR2ia}
{\rm (i)}
Assume the conditions of Theorem \ref{theo_GRi}, with $A_j(\cdot,\xi)\in{C}([0,T])$ replaced by
$A_j(\cdot,\xi)\in L^{\infty}([0,T])$, $j=1,\ldots,m$. Then the statement
remains true provided that we replace the conclusion
$u\in C^m([0,T];G^s(\R^n))$ by
\[
u\in C^{m-1}([0,T];G^s(\R^n))\cap W^{\infty,m}([0,T];G^s(\R^n)).
\]

{\rm (ii)}
Assume the conditions of Theorem \ref{theo_GR2i} with $A_j(\cdot,\xi)\in{C}([0,T])$ replaced by
$A_j(\cdot,\xi)\in L^{\infty}([0,T])$, $j=1,\ldots,m$. Then the
$C^{\infty}$ well-posedness
remains true provided that we replace the conclusion
$u\in C^m([0,T];C^\infty(\R^n))$ by
\[
u\in C^{m-1}([0,T];C^\infty(\R^n))\cap W^{\infty,m}([0,T];C^\infty(\R^n)).
\]
\end{thm}
We refer to Remark \ref{REM:sh} for a brief discussion of the strictly
hyperbolic case. In this case, even in the situation of the low regularity of
coefficients ($C^{1}$), one can analyse the global behaviour of solutions
with respect to time (see \cite{MR}). The cases of constant coefficients and
systems with controlled oscillations have been treated in \cite{RS} and
\cite{RW}, respectively.

Finally, we describe the contents of the sections in more details.

Section \ref{SEC:ex} collects some motivating examples of applications of our results.
In Section \ref{SEC:qs} we recall the required facts about the quasi-symmetriser
and in Section \ref{SEC:se} we use it to derive the energy estimate for the solutions
of the hyperbolic system in Sylvester form corresponding to the Cauchy
problem \eqref{CP}. The estimate on the part of the energy
corresponding to lower order terms is given in Section \ref{SEC:lc}. In Section \ref{SEC:wp}
we prove Theorems \ref{theo_GRi}, \ref{theo_GR2i} and \ref{theo_GR2ia} and we end the paper
with a final remark on the Levi conditions \eqref{EQ:lot}.


\section{Examples}
\label{SEC:ex}

Let us first give an example of the Levi conditions \eqref{EQ:lot} for the equations
of third order, $m=3$. In this case \eqref{EQ:lot} become
\beq
\begin{split}
\label{LC3i}
&|A_3-A_{(3)}|^2\le C(\lambda_1^2\lambda_2^2+\lambda_2^2\lambda^2_3+\lambda_3^2\lambda_1^2),\\
&|A_2-A_{(2)}|^2\le C((\lambda_1+\lambda_2)^2+(\lambda_2+\lambda_3)^2+(\lambda_3+\lambda_1)^2),\\
&|A_1-A_{(1)}|^2\prec C,
\end{split}
\eeq
for some $C>0$. It is convenient in certain applications, whenever possible, to write conditions
\eqref{LC} and \eqref{EQ:lot} in terms of the coefficients of the equation. Such
analysis for \eqref{CP} has been recently carried out by Jannelli and Taglialatela
\cite{JT}. In Example 3 below we will give an example of the meaning of conditions
\eqref{LC3i}.

For the future technicality, similarly to \eqref{LC3i}, we may
also use an equivalent formulation of \eqref{EQ:lot} as
\begin{equation}\label{EQ:lot2}
\left|(A_{m-j+1}-A_{(m-j+1)})(t,\xi)\right|^2 \leq
C\sum_{i=1}^m |
\sum_{\substack{1\leq \ell_1<\cdots<\ell_{m-j}\leq m \\ \ell_h\not=i\; \forall h}}
\lambda_{\ell_1}(t,\xi)\cdots \lambda_{\ell_{m-j}}(t,\xi)|^2.
\end{equation}

Condition \eqref{LC} can be often reformulated in terms of the
discriminant of \eqref{CP} defined by
$\Delta(t,\xi)=\prod_{i<j} (\lambda_i(t,\xi)-\lambda_j(t,\xi))^2.$
Thus, for $m=2$, $n=1$, and the equation
$$
\partial_t^2 u+a_1(t)\partial_t\partial_x u+a_2(t)\partial_x^2 u=0,
$$
condition \eqref{LC} is equivalent to the existence of $c>0$ such that
$\Delta(t)\geq c a_1(t)^2$,
where
$\Delta(t,\xi)=\Delta(t)\xi$ and $\Delta(t)=a_1^2(t)-4a_2(t)\geq 0$ is the
condition of the hyperbolicity.

For $m=3$, $n=1$, and the equation
$$
\partial_t^3 u+a_1(t)\partial_x\partial_t^2 u+a_2(t)\partial_x^2\partial_t u
+a_3(t)\partial_x^3 u=0,
$$
following \cite{KS}, we have
$\Delta(t,\xi)=\Delta(t)\xi$, with
$\Delta(t)=-4a_2^3-27 a_3^2+a_1^2 a_2^2-4a_1^3 a_3+18a_1 a_2 a_3\geq 0$,
and \eqref{LC} is equivalent to
$\Delta(t)\geq c(a_1(t) a_2(t)-9 a_3(t))^2.$

Since the hyperbolic equations above have homogeneous symbols, the
coefficients are real.
We refer to Colombini--Orr\'u \cite{CO} and Kinoshita--Spagnolo \cite{KS}
for more examples of equations without lower order terms in one dimension
$n=1$.

We now give more examples, which correspond to the new possibility, ensured by
Theorems \ref{theo_GRi} and \ref{theo_GR2i}, to consider equations with lower
order terms and equations in higher dimensions $n\geq 1$.

\subsection*{Example 1}
As a first example we consider the second order equation
\[
D_t^2 u+a_{2,2}(t)D_x^2u+a_{2,1}(t)D_xu+a_{1,0}(t)D_tu+a_{2,0}(t)u=0,
\qquad \text{$t\in[0,T]$ and $x\in\R$}.
\]
Assume $a_{2,2}(t)$ is real and
$a_{2,2}(t)\le 0$. The condition \eqref{LC} is trivially satisfied by the roots
\[
\begin{split}
\lambda_1(t,\xi)&=-\sqrt{-a_{2,2}(t)}|\xi|,\\
\lambda_2(t,\xi)&=+\sqrt{-a_{2,2}(t)}|\xi|.
\end{split}
\]
The well-posedness results of Section 1 are obtained under the conditions \eqref{EQ:lot} on the lower order terms. In this case \eqref{EQ:lot} means that the coefficient $a_{1,0}(t)$ is bounded on $[0,T]$ and that there exists a constant $c>0$ such that $|a_{2,1}(t)\xi+a_{2,0}(t)|^2\le -c a_{2,2}(t)\xi^2$ for all $t\in[0,T]$ and for $\xi$ away from $0$. Note that this last condition holds if $|a_{2,1}(t)|^2+|a_{2,0}(t)|^2\le -c'a_{2,2}(t)$ for some $c'>0$ on the $t$-interval $[0,T]$.

Take now the general second order equation
\[
D_t^2 u+a_{1,1}(t)D_xD_tu+a_{2,2}(t)D_x^2u+a_{2,1}(t)D_xu+a_{1,0}(t)D_tu+a_{2,0}(t)u=0.
\]
As observed above condition \eqref{LC} coincides with the bound from below
\[
\Delta(t)=a_{1,1}^2(t)-4a_{2,2}(t)\ge c_0 a_{1,1}^2(t),
\]
valid for some $c_0>0$ on $[0,T]$ (\cite[(15)]{KS}).
Here $a_{1,1}, a_{2,2}$ are assumed real.
The conditions \eqref{EQ:lot} on the lower order terms are of the type $|a_{1,0}(t)|\le c_1$ and $|a_{2,1}(t)\xi+a_{2,0}(t)|^2\le c_2(4a_{1,1}^2(t)-8a_{2,2}(t))\xi^2$ for all $t\in[0,T]$ and for $\xi$ away from $0$.

\subsection*{Example 2}
The equation
\[
D_t^2 u+\sum_{j=1}^n a_{1,j}(t)D_{x_j}D_tu+a_{2}(t)\sum_{j=1}^nD_{x_j}^2u+\sum_{j=1}^n b_j(t)D_{x_j}u +b(t)D_tu+d(t)u=0
\]
is an $n$-dimensional version of the previous example, with real
$a_{1,j}$ and $a_{2}$. The condition \eqref{LC} is trivially satisfied
when $a_2(t)\le 0$. The conditions \eqref{EQ:lot} on the lower order
terms are as follows:
\[
\begin{split}
|b(t)|&\le c,\\
\biggl|\sum_{j=1}^n b_j(t)\xi_j+d(t)\biggr|^2&\le c\biggl[4\biggl(\sum_{j=1}^n a_{1,j}(t)\xi_j\biggr)^2-8a_2(t)|\xi|^2\biggr],
\end{split}
\]
for $t\in[0,T]$ and $\xi$ away from $0$.

\subsection*{Example 3}
We finally give an example of a higher order equation. Let
\begin{multline*}
D_t^3u-(a+b+c)D_xD_t^2u+(ab+ac+bc)D_x^2D_tu-abcD_x^3u\\
+\sum_{l<3}a_{3,l}(t)D^l_xu+\sum_{l<2}a_{2,l}(t)D_x^l D_tu+a_{1,0}(t)D_t^2u=0,
\end{multline*}
where $a(t)$, $b(t)$ and $c(t)$ are real-valued functions with $b$ and $c$ bounded above and from below by $a$ (for instance, $1/4a(t)\le b(t)\le 1/2 a(t)$  and $1/16 a(t)\le c(t)\le 1/8 a(t)$ for all $t\in[0,T]$). It follows that condition \eqref{LC} on the roots $\lambda_1(t,\xi)=a(t)\xi$, $\lambda_2(t,\xi)=b(t)\xi$ and $\lambda_3(t,\xi)=c(t)\xi$ is fulfilled on $[0,T]$ for all $\xi\in\R$. The Levi conditions \eqref{EQ:lot}
on the lower order terms are of the following type:
\[
\begin{split}
|a_{3,2}(t)\xi^2+a_{3,1}(t)\xi+a_{3,0}(t)|^2&\le c\, 
 a^4(t)\xi^4,\\
|a_{2,1}(t)\xi+a_{2,0}(t)|^2&\le 
c\, a^2(t)\xi^2,\\
|a_{1,0}(t)|^2&\le c,
\end{split}
\]
for $t\in[0,T]$ and $\xi$ away from $0$.

\section{The quasi-symmetriser}
\label{SEC:qs}

We begin by recalling a few facts concerning the quasi-symmetriser. For more details see \cite{DS, KS}. Note that for $m\times m$ matrices $A_1$ and $A_2$ the notation $A_1\le A_2$ means $(A_1v,v)\le (A_2v,v)$ for all $v\in\C^m$ with $(\cdot,\cdot)$ the scalar product in $\C^m$.
Let $A(\lambda)$ be the $m\times m$ Sylvester matrix with real eigenvalues $\lambda_j$, i.e.,
\[
A(\lambda)=\left(
    \begin{array}{ccccc}
      0 & 1 & 0 & \dots & 0\\
      0 & 0 & 1 & \dots & 0 \\
      \dots & \dots & \dots & \dots & 1 \\
      -\sigma_m^{(m)}(\lambda) & -\sigma_{m-1}^{(m)}(\lambda) & \dots & \dots & -\sigma_1^{(m)}(\lambda) \\
    \end{array}
  \right),
\]
where
\[
\sigma_h^{(m)}(\lambda)=(-1)^h\sum_{1\le i_1<...<i_h\le m}\lambda_{i_1}...\lambda_{i_h}
\]
for all $1\le h\le m$. In the sequel we make use of the following notations: $\mathcal{P}_m$ for the class of permutations of $\{1,...,m\}$, $\lambda_\rho=(\lambda_{\rho_1},...,\lambda_{\rho_m})$ with $\lambda\in\R^m$ and $\rho\in\mathcal{P}_m$, $\pi_i\lambda=(\lambda_1,...,\lambda_{i-1},\lambda_{i+1},...,\lambda_m)$ and $\lambda'=\pi_m\lambda=(\lambda_1,...,\lambda_{m-1})$. Following Section 4 in \cite{KS} we have that the quasi-symmetriser is the Hermitian matrix
\[
Q_\eps(\lambda)=\sum_{\rho\in\mathcal{P}_m} P_\eps^{(m)}(\lambda_\rho)^\ast P_\eps^{(m)}(\lambda_\rho),
\]
where $\eps\in(0,1]$, $P_\eps^{(m)}(\lambda)=H^{(m)}_\eps P^{(m)}(\lambda)$, $H_\eps^{(m)}={\rm diag}\{\eps^{m-1},...,\eps,1\}$ and the matrix $P^{(m)}(\lambda)$ is defined inductively by $P^{(1)}(\lambda)=1$ and
\[
P^{(m)}(\lambda)=\left(
    \begin{array}{ccccc}
      \, & \, & \, & \, & 0\\
      \, & \, & P^{(m-1)}(\lambda') & \, & \vdots \\
      \, & \, & \, & \, & 0 \\
      \sigma_{m-1}^{(m-1)}(\lambda') & \dots & \dots & \sigma_1^{(m-1)}(\lambda') & 1 \\
    \end{array}
  \right).
\]
Note that $P^{(m)}(\lambda)$ is depending only on $\lambda'$. Finally, let $W^{(m)}_i(\lambda)$ denote the row vector
\[
\big(\sigma_{m-1}^{(m-1)}(\pi_i\lambda),...,\sigma_1^{(m-1)}(\pi_i\lambda),1\big),\quad 1\le i\le m,
\]
and let $\mathcal{W}^{(m)}(\lambda)$ be the matrix with row vectors $W^{(m)}_i$. The following proposition collects the main properties of the quasi-symmetriser $Q^{(m)}_\eps(\lambda)$. For a detailed proof we refer the reader to Propositions 1 and 2 in \cite{KS} and to Proposition 1 in \cite{DS}.
\begin{proposition}
\label{prop_qs}
\leavevmode
\begin{itemize}
\item[(i)] The quasi-symmetriser $Q_\eps^{(m)}(\lambda)$ can be written as
\[
Q_0^{(m)}(\lambda)+\eps^2 Q_1^{(m)}(\lambda)+...+\eps^{2(m-1)}Q_{m-1}^{(m)}(\lambda),
\]
where the matrices $Q^{(m)}_i(\lambda)$, $i=1,...,m-1,$ are nonnegative and Hermitian with
entries being symmetric polynomials in $\lambda_1,...,\lambda_m$.
\item[(ii)] There exists a function $C_m(\lambda)$ bounded for
bounded $|\lambda|$ such that
\[
C_m(\lambda)^{-1}\eps^{2(m-1)}I\le Q^{(m)}_\eps(\lambda)\le C_m(\lambda)I.
\]
\item[(iii)] We have
\[
-C_m(\lambda)\eps Q_\eps^{(m)}(\lambda)\le Q_\eps^{(m)}(\lambda)A(\lambda)-A(\lambda)^\ast Q_\eps^{(m)}(\lambda)\le C_m(\lambda)\eps Q_\eps^{(m)}(\lambda).
\]
\item[(iv)] For any $(m-1)\times(m-1)$ matrix $T$ let $T^\sharp$ denote the $m\times m$ matrix
\[
\left(
    \begin{array}{cc}
    T & 0\\
    0 & 0 \\
    \end{array}
  \right).
\]
Then, $Q_\eps^{(m)}(\lambda)=Q_0^{(m)}(\lambda)+\eps^2\sum_{i=1}^m Q_\eps^{(m-1)}(\pi_i\lambda)^\sharp$.
\item[(v)] We have
\[
Q_0^{(m)}(\lambda)=(m-1)!\mathcal{W}^{(m)}(\lambda)^\ast \mathcal{W}^{(m)}(\lambda).
\]
\item[(vi)] We have
\[
\det Q_0^{(m)}(\lambda)=(m-1)!\prod_{1\le i<j\le m}(\lambda_i-\lambda_j)^2.
\]
\item[(vii)] There exists a constant $C_m$ such that
\[
q_{0,11}^{(m)}(\lambda)\cdots q_{0,mm}^{(m)}(\lambda)\le C_m\prod_{1\le i<j\le m}(\lambda^2_i+\lambda^2_j).
\]
\end{itemize}
\end{proposition}
We finally recall that a family $\{Q_\alpha\}$ of nonnegative Hermitian matrices is called \emph{nearly diagonal} if there exists a positive constant $c_0$ such that
\[
Q_\alpha\ge c_0\,{\rm diag}\,Q_\alpha
\]
for all $\alpha$, with ${\rm diag}\,Q_\alpha ={\rm diag}\{q_{\alpha,11},...,q_{\alpha, mm}\}$. The following linear algebra result is proven in \cite[Lemma 1]{KS}.
\begin{lem}
\label{lem_old}
Let $\{Q_\alpha\}$ be a family of nonnegative Hermitian $m\times m$ matrices such that $\det Q_\alpha>0$ and
\[
\det Q_\alpha \ge c\, q_{\alpha,11}q_{\alpha,22}\cdots q_{\alpha,mm}
\]
for a certain constant $c>0$ independent of $\alpha$. Then,
\[
Q_\alpha\ge c\, m^{1-m}\,{\rm diag}\,Q_\alpha
\]
for all $\alpha$, i.e., the family $\{Q_\alpha\}$ is nearly diagonal.
\end{lem}
Lemma \ref{lem_old} is employed to prove that the family  $Q_\eps^{(m)}(\lambda)$ of quasi-symmetrisers defined above is nearly diagonal when $\lambda$ belongs to a suitable set. The following statement is proven in \cite[Proposition 3]{KS}.
\begin{proposition}
\label{prop_SM}
For any $M>0$ define the set
\[
\mathcal{S}_M=\{\lambda\in\R^m:\, \lambda_i^2+\lambda_j^2\le M (\lambda_i-\lambda_j)^2,\quad 1\le i<j\le m\}.
\]
Then the family of matrices $\{Q_\eps^{(m)}(\lambda):\, 0<\eps\le 1, \lambda\in\mathcal{S}_M\}$ is nearly diagonal.
\end{proposition}
We conclude this section with a result on nearly diagonal matrices depending on 3 parameters (i.e. $\eps, t, \xi$)
which will be crucial in the next section. Note that this is a straightforward
extension of Lemma 2 in \cite{KS} valid for 2 parameter
(i.e. $\eps, t$) dependent matrices.
\begin{lem}
\label{lem_new}
Let $\{ Q_\eps(t,\xi): 0<\eps\le 1, 0\le t\le T, \xi\in\R^n\}$ be a nearly diagonal family of coercive Hermitian matrices of class ${C}^k$ in $t$, $k\ge 1$. Then, there exists a constant $C_T>0$ such that for any function $V:[0,T]\times\R^n\to \C^m$
we have
\[
\int_{0}^T \frac{|(\partial_t Q_\eps(t,\xi) V(t,\xi),V(t,\xi))|}{(Q_\eps(t,\xi)V(t,\xi),V(t,\xi))^{1-1/k}|V(t,\xi)|^{2/k}}\, dt\le C_T \sup_{\xi\in\R^n}\Vert Q_\eps(\cdot,\xi)\Vert^{1/k}_{{C}^k([0,T])}.
\]
\end{lem}

\section{Reduction to a first order system and energy estimate}
\label{SEC:se}

We now go back to the Cauchy problem \eqref{CP} and perform a reduction of the $m$-order equation to a first order system as in \cite{Taylor:81}. Let  $\lara{D_x}$ be
the pseudo-differential operator with symbol $\lara{\xi}=(1+|\xi|^2)^{\frac{1}{2}}$. The transformation
\[
u_l=D_t^{l-1}\lara{D_x}^{m-l}u,
\]
with $l=1,...,m$, makes the Cauchy problem \eqref{CP} equivalent to the following system
\beq
\label{syst_Taylor}
D_t\left(
                             \begin{array}{c}
                               u_1 \\
                               \cdot \\
                               \cdot\\
                               u_m \\
                             \end{array}
                           \right)
= \left(
    \begin{array}{ccccc}
      0 & \lara{D_x} & 0 & \dots & 0\\
      0 & 0 & \lara{D_x} & \dots & 0 \\
      \dots & \dots & \dots & \dots & \lara{D_x} \\
      b_1 & b_2 & \dots & \dots & b_m \\
    \end{array}
  \right)
  \left(\begin{array}{c}
                               u_1 \\
                               \cdot \\
                               \cdot\\
                               u_m \\
                             \end{array}
                           \right),
\eeq
where
\[
b_j=-A_{m-j+1}(t,D_x)\lara{D_x}^{j-m},
\]
with initial condition
\beq
\label{ic_Taylor}
u_l|_{t=0}=\lara{D_x}^{m-l}g_l,\qquad l=1,...,m.
\eeq
The matrix in \eqref{syst_Taylor} can be written as $A_1+B$ with
\[
A_1=\left(
    \begin{array}{ccccc}
      0 & \lara{D_x} & 0 & \dots & 0\\
      0 & 0 & \lara{D_x} & \dots & 0 \\
      \dots & \dots & \dots & \dots & \lara{D_x} \\
      b_{(1)} & b_{(2)} & \dots & \dots & b_{(m)} \\
    \end{array}
  \right),
\]
where $b_{(j)}=-A_{(m-j+1)}(t,D_x)\lara{D_x}^{j-m}$ is the principal part of the operator $b_{j}=-A_{m-j+1}(t,D_x)\lara{D_x}^{j-m}$ and
\[
B=\left(
    \begin{array}{ccccc}
      0 & 0 & 0 & \dots & 0\\
      0 & 0 & 0& \dots & 0 \\
      \dots & \dots & \dots & \dots & 0 \\
      b_1-b_{(1)} & b_2-b_{(2)} & \dots & \dots & b_m-b_{(m)} \\
    \end{array}
  \right).
\]
By Fourier transforming both sides of \eqref{syst_Taylor} in $x$ we obtain the system
\beq
\label{system_new}
\begin{split}
D_t V&=A_1(t,\xi)V+B(t,\xi)V,\\
V|_{t=0}(\xi)&=V_0(\xi),
\end{split}
\eeq
where $V$ is the $m$-column with entries $v_l=\widehat{u}_l$, $V_0$ is the $m$-column with entries
$v_{0,l}=\lara{\xi}^{m-l}\widehat{g}_l$ and
\begin{multline}
\label{mA}
A_1(t,\xi)=\left(
    \begin{array}{ccccc}
      0 & \lara{\xi} & 0 & \dots & 0\\
      0 & 0 & \lara{\xi} & \dots & 0 \\
      \dots & \dots & \dots & \dots & \lara{\xi} \\
      b_{(1)}(t,\xi) & b_{(2)}(t,\xi) & \dots & \dots & b_{(m)}(t,\xi) \\
    \end{array}
  \right),\\[0.3cm]
  b_{(j)}(t,\xi)=-A_{(m-j+1)}(t,\xi)\lara{\xi}^{j-m},
\end{multline}
\begin{multline*}
B(t,\xi)=\left(
    \begin{array}{ccccc}
      0 & 0 & 0 & \dots & 0\\
      0 & 0 & 0& \dots & 0 \\
      \dots & \dots & \dots & \dots & 0 \\
      (b_1-b_{(1)})(t,\xi) & \dots & \dots & \dots & (b_m-b_{(m)})(t,\xi) \\
    \end{array}
  \right),\\[0.3cm]
(b_j-b_{(j)})(t,\xi)=-(A_{m-j+1}-A_{(m-j+1)})(t,\xi)\lara{\xi}^{j-m}.
\end{multline*}
From now on we will concentrate on the system \eqref{system_new} and on the matrix $$A(t,\xi):=\lara{\xi}^{-1}A_1(t,\xi)$$ for which we will construct a quasi-symmetriser. Note that the eigenvalues of the matrix $A_1$ are exactly the roots $\lambda_j(t,\xi)$, $j=1,...,m$. It is clear that the condition \eqref{LC} holds for the eigenvalues
$\lara{\xi}^{-1}\lambda_j(t,\xi)$
of the $0$-order matrix $A(t,\xi)$ as well.
Let us define the energy
\[
E_\eps(t,\xi)=(Q^{(m)}_\eps(t,\xi) V(t,\xi), V(t,\xi)).
\]
We have
\[
\begin{split}
\partial_t E_\eps(t,\xi)&=(\partial_tQ^{(m)}_\eps V,V)+ i(Q^{(m)}_\eps D_tV,V)-i(Q^{(m)}_\eps V,D_tV)\\
&=(\partial_tQ^{(m)}_\eps V,V)+i(Q^{(m)}_\eps(A_1V+BV),V)-i(Q^{(m)}_\eps V,A_1V+BV)\\
&=(\partial_tQ^{(m)}_\eps V,V)+i\lara{\xi}((Q^{(m)}_\eps A-A^\ast Q^{(m)}_\eps)V,V)+((Q^{(m)}_\eps B-B^\ast Q^{(m)}_\eps)V,V).
\end{split}
\]
It follows that
\begin{multline}
\label{EE}
\partial_t E_\eps(t,\xi)\le \frac{|(\partial_tQ^{(m)}_\eps V,V)|E_\eps}{(Q^{(m)}_\eps(t,\xi) V(t,\xi), V(t,\xi))}+|\lara{\xi}((Q^{(m)}_\eps A-A^\ast Q^{(m)}_\eps)V,V)|+\\
+|((Q^{(m)}_\eps B-B^\ast Q^{(m)}_\eps)V,V)|.
\end{multline}
We recall that from Proposition \ref{prop_qs} $Q_\eps^{(m)}(t,\xi)$ is a family of smooth nonnegative Hermitian matrices such that
\beq
\label{p1}
Q_\eps^{(m)}(t,\xi)=Q_0^{(m)}(t,\xi)+\eps^2 Q_1^{(m)}(t,\xi)+...+\eps^{2(m-1)}Q_{m-1}^{(m)}(t,\xi).
\eeq
In addition there exists a constant $C_m>0$ such that for all $t\in[0,T]$, $\xi\in\R^n$ and $\eps\in(0,1]$ the following estimates hold uniformly in $V$:
\beq
\label{p2}
C_m^{-1}\eps^{2(m-1)}|V|^2\le (Q^{(m)}_\eps(t,\xi)V,V)\le C_m|V|^2,
\eeq
\beq
\label{p3}
|((Q_\eps^{(m)}A-A^\ast Q_\eps^{(m)})(t,\xi)V,V)|\le C_m\eps (Q_\eps^{(m)}(t,\xi)V,V).
\eeq
Finally, condition \eqref{LC} and Proposition \ref{prop_SM} ensure that the family
$\{ Q_\eps^{(m)}(t,\xi):\, \eps\in(0,1],\, t\in[0,T],\, \xi\in\R^n\}$ is nearly diagonal.

In the sequel we assume that the coefficients $a_j$ in the equation \eqref{CP} are of class ${C}^k$, or in other words that the matrix $A(t,\xi)$ has entries of class ${C}^k$ in $t\in[0,T]$. It follows by construction that the quasi-symmetriser has the same regularity property.
We now estimate the three terms of the right hand side of
\eqref{EE}.

\subsection{First term}

We write $\frac{|(\partial_tQ^{(m)}_\eps V,V)|}{(Q^{(m)}_\eps(t,\xi) V(t,\xi), V(t,\xi))}$ as
\[
\frac{|(\partial_tQ^{(m)}_\eps V,V)|}{(Q^{(m)}_\eps(t,\xi) V(t,\xi), V(t,\xi))^{1-1/k}(Q^{(m)}_\eps(t,\xi) V(t,\xi), V(t,\xi))^{1/k}}.
\]
From \eqref{p2} we have
\begin{multline*}
\frac{|(\partial_tQ^{(m)}_\eps V,V)|}{(Q^{(m)}_\eps(t,\xi) V(t,\xi), V(t,\xi))}\le \frac{|(\partial_tQ^{(m)}_\eps V,V)|}{(Q^{(m)}_\eps(t,\xi) V(t,\xi), V(t,\xi))^{1-1/k}(C_m^{-1}\eps^{2(m-1)}|V|^2)^{1/k}}\\
\le C_m^{1/k}\eps^{-2(m-1)/k}\frac{|(\partial_tQ^{(m)}_\eps V,V)|}{(Q^{(m)}_\eps(t,\xi) V(t,\xi), V(t,\xi))^{1-1/k}|V|^{2/k}}.
\end{multline*}
An application of Lemma \ref{lem_new} yields the estimate
\begin{multline*}
\int_{0}^T\frac{|(\partial_tQ^{(m)}_\eps V,V)|}{(Q^{(m)}_\eps(t,\xi) V(t,\xi), V(t,\xi))}\, dt\le C_m^{1/k}\eps^{-2(m-1)/k}C_T\sup_{\xi\in\R^n}\Vert Q_\eps(\cdot,\xi)\Vert^{1/k}_{{C}^k([0,T])}\\
\le C_1\eps^{-2(m-1)/k},
\end{multline*}
for all $\eps\in(0,1]$. Setting $\frac{|(\partial_tQ^{(m)}_\eps V,V)|}{(Q^{(m)}_\eps(t,\xi) V(t,\xi), V(t,\xi))}=K_\eps(t,\xi)$ we conclude that
\[
\frac{|(\partial_tQ^{(m)}_\eps V,V)|E_\eps}{(Q^{(m)}_\eps(t,\xi) V(t,\xi), V(t,\xi))}=K_\eps(t,\xi)E_\eps,
\]
with
\[
\int_{0}^T K_\eps(t,\xi)\, dt\le C_1\eps^{-2(m-1)/k}.
\]
\subsection{Second term}
From the property \eqref{p3} we have that
\[
|\lara{\xi}((Q^{(m)}_\eps A-A^\ast Q^{(m)}_\eps)V,V)|\le C_m\eps\lara{\xi}(Q_\eps^{(m)}(t,\xi)V,V)\le C_2\eps\lara{\xi}E_\eps.
\]

\subsection{Third term}
We now concentrate on
\[
((Q^{(m)}_\eps B-B^\ast Q^{(m)}_\eps)V,V),
\]
which is the main task in this paper.
By Proposition \ref{prop_qs}(iv) and the definition of the matrix $B(t,\xi)$
we have that
\begin{multline*}
((Q^{(m)}_\eps B-B^\ast Q^{(m)}_\eps)V,V)=((Q_0^{(m)}B-B^\ast Q_0^{(m)})V,V)\\
+\eps^2\sum_{i=1}^m((Q^{(m-1)}_\eps(\pi_i\lambda)^\sharp B-B^\ast Q^{(m-1)}_\eps(\pi_i\lambda)^\sharp)V,V),
\end{multline*}
where we notice that
$(Q^{(m-1)}_\eps(\pi_i\lambda)^\sharp B-B^\ast Q^{(m-1)}_\eps(\pi_i\lambda)^\sharp)=0$
due to the structure of zeros in $B$ and in
$Q^{(m-1)}_\eps(\pi_i\lambda)^\sharp$. Hence
\[
((Q^{(m)}_\eps B-B^\ast Q^{(m)}_\eps)V,V)=((Q_0^{(m)} B-B^\ast Q_0^{(m)})V,V).
\]
Note that from Proposition \ref{prop_qs}(i) we have that $(Q_0 V,V)\le E_\eps$.
In the next section we will show that the conditions on $B$
corresponding to \eqref{EQ:lot} imply that
\beq
\label{cB}
|((Q_0^{(m)} B-B^\ast Q_0^{(m)})V,V)|\le C_3 (Q_0 V,V)\le C_3 E_\eps,
\eeq
for some constant $C_3>0$ independent of $t\in[0,T]$, $\xi\in\R^n$ and $V\in\C^m$.

\begin{remark}\label{REM:sh}
Note that condition \eqref{LC} is trivially satisfied when the roots are distinct, i.e. in the strictly hyperbolic case. It follows that the family $\{Q_\eps(\lambda)\}$ of quasi-symmetrisers is nearly diagonal and, therefore, there exists a constant $c_0>0$ such that $Q^{(m)}_0\ge c_0\text{diag}\,Q_0^{(m)}$. This means that
\[
(Q_0^{(m)}V,V)\ge c_0\sum_{i=1}^m q_{0,ii}|V_i|^2
\]
holds for all $V\in\C^m$. From the hypothesis of strict hyperbolicity it easily follows that
\[
\inf_{t\in[0,T], |\xi|\ge 1, i=1,...,m} q_{0,ii}(t,\xi)>0.
\]
This bound from below implies
\beq
\label{struct_b}
(Q_0^{(m)}(t,\xi)V,V)\ge c'_0 |V|^2
\eeq
for $t\in[0,T]$ and $|\xi|\ge 1$ and hence the estimate
\[
|((Q_0^{(m)} B-B^\ast Q_0^{(m)})V,V)|\le C_3 (Q_0 V,V)
\]
holds trivially in the strictly hyperbolic case for any lower order term $B$ (for our purposes it will not be restrictive to assume $|\xi|\ge 1$). Concluding, when the roots $\lambda_i$ are distinct the Gevrey and ultradistributional well-posedness results in Theorem \ref{theo_GRi} and Theorem \ref{theo_GR2i} can be stated without additional conditions on the lower order terms.
Strictly hyperbolic equations under low regularity (H\"older ${C}^{\alpha}$,
$0<\alpha<1$) of the coefficients have
been analysed by the authors in \cite{GR:11}, to which we refer for general statements
on the Gevrey and ultradistributional well-posedness in this setting.
\end{remark}

\section{Estimates for the lower order terms}
\label{SEC:lc}

We begin by rewriting $((Q_0^{(m)} B-B^\ast Q_0^{(m)})V,V)$ in terms of the matrix $\mathcal{W}=\mathcal{W}^{(m)}$. From Proposition \ref{prop_qs}(v) we have
\begin{multline*}
((Q_0^{(m)} B-B^\ast Q_0^{(m)})V,V)=(m-1)!((\mathcal{W}BV,\mathcal{W}V)-(\mathcal{W}V,\mathcal{W}BV))\\
=2i(m-1)!\Im (\mathcal{W}BV,\mathcal{W}V).
\end{multline*}
It follows that
\[
|((Q_0^{(m)} B-B^\ast Q_0^{(m)})V,V)|\le 2(m-1)!|\mathcal{W}BV||\mathcal{W}V|.
\]
Since
\[
(Q_0 V,V)=(m-1)!|\mathcal{W}V|^2
\]
we have that if
\beq
\label{cB2}
|\mathcal{W}BV|\le C|\mathcal{W}V|
\eeq
for some constant $C>0$ independent of $t$, $\xi$ and $V$, then the condition \eqref{cB} will hold. \emph{It is our task to show that the condition \eqref{EQ:lot} on the matrix $B$ of the lower order terms implies the estimate \eqref{cB2}}.

Before dealing with the general case of $B$ $m\times m$-matrix,
let us consider the instructive case $m=3$, which will illustrate the
general argument in a simplified setting. In the sequel, for $f$ and $g$ real-valued functions (in the variable $y$) we write $f(y)\prec g(y)$ if there exists a constant $C>0$ such that $f(y)\le C g(y)$ for all $y$. More precisely, we will set $y=(t,\xi)$ or $y=(t,\xi,V)$.

\subsection{The case $m=3$}

By definition of the row vectors $W^{(3)}_i$, $i=1,2,3$, we have that
\[
\mathcal{W}=\left(
    \begin{array}{ccc}
      \lambda_2\lambda_3 & -\lambda_2-\lambda_3 & 1\\
      \lambda_3\lambda_1 & -\lambda_3-\lambda_1 & 1\\
      \lambda_1\lambda_2 & -\lambda_1-\lambda_2 & 1
      \end{array}
  \right),
\]
where $\lambda_i$, $i=1,2,3$, are the 0-order normalised roots. Hence
\[
\mathcal{W}BV=\left(
    \begin{array}{c}
      B_1V_1+B_2V_2+B_3V_3\\
      B_1V_1+B_2V_2+B_3V_3\\
      B_1V_1+B_2V_2+B_3V_3
      \end{array}
  \right)
\]
and
\[
\mathcal{W}V=\left(
    \begin{array}{c}
      (\lambda_2\lambda_3)V_1-(\lambda_2+\lambda_3)V_2+V_3\\
      (\lambda_3\lambda_1)V_1-(\lambda_3+\lambda_1)V_2+V_3\\
      (\lambda_1\lambda_2)V_1-(\lambda_1+\lambda_2)V_2+V_3
      \end{array}
  \right).
\]
Thus, instead of working on proving \eqref{cB2} we can work on the equivalent inequality
\begin{multline}
\label{m3}
|B_1V_1+B_2V_2+B_3V_3|^2\prec |(\lambda_2\lambda_3)V_1-(\lambda_2+\lambda_3)V_2+V_3|^2+\\
+|(\lambda_3\lambda_1)V_1-(\lambda_3+\lambda_1)V_2+V_3|^2 + |(\lambda_1\lambda_2)V_1-(\lambda_1+\lambda_2)V_2+V_3|^2.
\end{multline}
In terms of the coefficients of the matrix $B$ the Levi conditions \eqref{LC3i}
on the lower order
terms can be written as
\beq
\begin{split}
\label{LC3}
&|B_1|^2\prec\lambda_1^2\lambda_2^2+\lambda_2^2\lambda^2_3+\lambda_3^2\lambda_1^2,\\
&|B_2|^2\prec (\lambda_1+\lambda_2)^2+(\lambda_2+\lambda_3)^2+(\lambda_3+\lambda_1)^2,\\
&|B_3|^2\prec c.
\end{split}
\eeq
Under these conditions we now want to prove that \eqref{m3} holds for all vectors $V$.
We note here that actually for the right hand side of \eqref{m3} by the triangle
inequality
we have the upper bound
\begin{multline*}
|(\lambda_2\lambda_3)V_1-(\lambda_2+\lambda_3)V_2+V_3|^2+
|(\lambda_3\lambda_1)V_1-(\lambda_3+\lambda_1)V_2+V_3|^2 +
|(\lambda_1\lambda_2)V_1-(\lambda_1+\lambda_2)V_2+V_3|^2\\
\prec(\lambda_1^2\lambda_2^2+\lambda_2^2\lambda^2_3+\lambda_3^2\lambda_1^2)|V_1|^2
+((\lambda_1+\lambda_2)^2+(\lambda_2+\lambda_3)^2+(\lambda_3+\lambda_1)^2)|V_2|^2+|V_3|^2,
\end{multline*}
in which the right hand side of \eqref{LC3} appears naturally.

Our strategy is to proceed by 3 steps making use of the following partition of $\R^3$:
\[
\R^3=\Sigma_1^{\delta_1} \cup \big(\big(\Sigma_1^{\delta_1}\big)^{\rm{c}}\cap\Sigma_2^{\delta_2}\big)\cup \big(\big(\Sigma_1^{\delta_1}\big)^{\rm{c}}\cap\big(\Sigma_2^{\delta_2}\big)^{\rm{c}}\big),
\]
where
\begin{multline*}
\Sigma_1^{\delta_1}:=\\
\{ V\in\R^3:\,|V_3|^2+((\lambda_1+\lambda_2)^2+(\lambda_2+\lambda_3)^2+(\lambda_3+\lambda_1)^2)|V_2|^2\le \delta_1(\lambda_1^2\lambda_2^2+\lambda_2^2\lambda^2_3+\lambda_3^2\lambda_1^2)|V_1|^2\},
\end{multline*}
and
\[
\Sigma_2^{\delta_2}:=\{V\in\R^3:\,|V_3|^2\le\delta_2((\lambda_1+\lambda_2)^2+
(\lambda_2+\lambda_3)^2+(\lambda_3+\lambda_1)^2)|V_2|^2\}.
\]
\\
{\bf Estimate on $\Sigma_1^{\delta_1}$.}

\medskip
Making use of the conditions \eqref{LC3} we have that
\begin{multline}\label{aux1}
|B_1V_1+B_2V_2+B_3V_3|^2\prec |B_1|^2|V_1|^2+|B_2|^2|V_2|^2+|B_3|^2|V_3|^2\\
\prec (\lambda_1^2\lambda_2^2+\lambda_2^2\lambda^2_3+\lambda_3^2\lambda_1^2)|V_1|^2+
((\lambda_1+\lambda_2)^2+(\lambda_2+\lambda_3)^2+(\lambda_3+\lambda_1)^2)|V_2|^2+|V_3|^2\\
\prec (\lambda_1^2\lambda_2^2+\lambda_2^2\lambda^2_3+\lambda_3^2\lambda_1^2)|V_1|^2
\end{multline}
on $\Sigma_1^{\delta_1}$. Note that
\begin{multline*}
|(\lambda_2\lambda_3)V_1-(\lambda_2+\lambda_3)V_2+V_3|^2+
|(\lambda_3\lambda_1)V_1-(\lambda_3+\lambda_1)V_2+V_3|^2 \succ |(\lambda_2\lambda_3-\lambda_3\lambda_1)V_1-(\lambda_2-\lambda_1)V_2|^2\\
\succ (\lambda_2-\lambda_1)^2|\lambda_3 V_1-V_2|^2\succ (\lambda_1^2+\lambda_2^2)|\lambda_3 V_1-V_2|^2,
\end{multline*}
where also in the last line we make use of the condition \eqref{LC} on the
roots $\lambda_i$. Hence, by applying this to different combinations of terms, we get
\begin{multline}
\label{bound1}
|(\lambda_2\lambda_3)V_1-(\lambda_2+\lambda_3)V_2+V_3|^2
+|(\lambda_3\lambda_1)V_1-(\lambda_3+\lambda_1)V_2+V_3|^2 + |(\lambda_1\lambda_2)V_1-(\lambda_1+\lambda_2)V_2+V_3|^2 \\
\succ(\lambda_1^2+\lambda_2^2)|\lambda_3 V_1-V_2|^2+(\lambda_2^2+\lambda_3^2)|\lambda_1 V_1-V_2|^2+(\lambda_3^2+\lambda_1^2)|\lambda_2 V_1-V_2|^2\\
\succ \lambda_1^2(|\lambda_3 V_1-V_2|^2+|\lambda_2 V_1-V_2|^2)+\lambda_2^2(|\lambda_3 V_1-V_2|^2+|\lambda_1 V_1-V_2|^2)\\
+\lambda_3^2(|\lambda_2 V_1-V_2|^2+|\lambda_1 V_1-V_2|^2)\\
\succ (\lambda_1^2(\lambda_3-\lambda_2)^2+\lambda_2^2(\lambda_3-\lambda_1)^2+\lambda_3^2(\lambda_2-\lambda_1)^2)|V_1|^2\\
\succ (\lambda_1^2\lambda_2^2+\lambda_2^2\lambda^2_3+\lambda_3^2\lambda_1^2)|V_1|^2.
\end{multline}
From the bound from below \eqref{bound1} and the estimate \eqref{aux1} one has that the inequality \eqref{m3} holds true in the region $\Sigma_1^{\delta_1}$ for all $\delta_1>0$.\\
\\
{\bf Estimate on $\big(\Sigma_1^{\delta_1}\big)^{\rm{c}}\cap\Sigma_2^{\delta_2}.$}

\medskip
We assume from now on that $V\in\big(\Sigma_1^{\delta_1}\big)^{\rm{c}}$ which means that
\[
|V_3|^2+((\lambda_1+\lambda_2)^2+(\lambda_2+\lambda_3)^2+(\lambda_3+\lambda_1)^2)|V_2|^2> \delta_1(\lambda_1^2\lambda_2^2+\lambda_2^2\lambda^2_3+\lambda_3^2\lambda_1^2)|V_1|^2.
\]
One immediately has
\begin{multline*}
|B_1V_1+B_2V_2+B_3V_3|^2\prec |B_1|^2|V_1|^2+|B_2|^2|V_2|^2+|B_3|^2|V_3|^2\\
\prec (\lambda_1^2\lambda_2^2+\lambda_2^2\lambda^2_3+\lambda_3^2\lambda_1^2)|V_1|^2+
((\lambda_1+\lambda_2)^2+(\lambda_2+\lambda_3)^2+(\lambda_3+\lambda_1)^2)|V_2|^2+|V_3|^2\\
\prec ((\lambda_1+\lambda_2)^2+(\lambda_2+\lambda_3)^2+(\lambda_3+\lambda_1)^2)|V_2|^2+|V_3|^2.
\end{multline*}
More precisely,
\begin{equation}\label{aux2}
|B_1V_1+B_2V_2+B_3V_3|^2\prec ((\lambda_1+\lambda_2)^2+(\lambda_2+\lambda_3)^2+(\lambda_3+\lambda_1)^2)|V_2|^2
\end{equation}
holds for all $V\in\big(\Sigma_1^{\delta_1}\big)^{\rm{c}}\cap\Sigma_2^{\delta_2}$ .
We estimate the right-hand side of \eqref{m3} as
\begin{multline*}
|(\lambda_2\lambda_3)V_1-(\lambda_2+\lambda_3)V_2+V_3|^2+
|(\lambda_3\lambda_1)V_1-(\lambda_3+\lambda_1)V_2+V_3|^2 +
|(\lambda_1\lambda_2)V_1-(\lambda_1+\lambda_2)V_2+V_3|^2\\
\succ\gamma_1(|(\lambda_2+\lambda_3)V_2-V_3|^2+|(\lambda_3+\lambda_1)V_2-V_3|^2+|(\lambda_1+\lambda_2)V_2-V_3|^2)\\
-\gamma_2(\lambda_1^2\lambda_2^2+\lambda_2^2\lambda^2_3+\lambda_3^2\lambda_1^2)|V_1|^2.
\end{multline*}
By using condition \eqref{LC} we get the estimate
\begin{multline*}
(\lambda_2-\lambda_1)^2+(\lambda_3-\lambda_2)^2+(\lambda_3-\lambda_1)^2\ge\frac{2}{M}(\lambda_1^2+\lambda_2^2+\lambda_3^2)\\
\ge\frac{1}{2M}((\lambda_1+\lambda_2)^2+(\lambda_2+\lambda_3)^2+(\lambda_3+\lambda_1)^2)
\end{multline*}
and then
\begin{multline*}
|(\lambda_2\lambda_3)V_1-(\lambda_2+\lambda_3)V_2+V_3|^2+
|(\lambda_3\lambda_1)V_1-(\lambda_3+\lambda_1)V_2+V_3|^2 +
|(\lambda_1\lambda_2)V_1-(\lambda_1+\lambda_2)V_2+V_3|^2\\
\succ\gamma_1((\lambda_2-\lambda_1)^2+(\lambda_3-\lambda_2)^2+(\lambda_3-\lambda_1)^2)|V_2|^2-
\gamma_2(\lambda_1^2\lambda_2^2+\lambda_2^2\lambda^2_3+\lambda_3^2\lambda_1^2)|V_1|^2\\
\succ \gamma'_1((\lambda_1+\lambda_2)^2+(\lambda_2+\lambda_3)^2+(\lambda_3+\lambda_1)^2)|V_2|^2-
\gamma_2(\lambda_1^2\lambda_2^2+\lambda_2^2\lambda^2_3+\lambda_3^2\lambda_1^2)|V_1|^2\\
\succ (\gamma'_1-\gamma_2\frac{1}{\delta_1}(\delta_2+1))
((\lambda_1+\lambda_2)^2+(\lambda_2+\lambda_3)^2+(\lambda_3+\lambda_1)^2)|V_2|^2.
\end{multline*}
Combining this with \eqref{aux2} we conclude that for any $\delta_2$ and for
$\delta_1$ big enough the right-hand side of \eqref{m3}
can be estimated from below by $|V_2|^2$ and, therefore, \eqref{m3} holds true on $\big(\Sigma_1^{\delta_1}\big)^{\rm{c}}\cap\Sigma_2^{\delta_2}$.\\
\\
{\bf Estimate on $\big(\Sigma_1^{\delta_1}\big)^{\rm{c}}\cap\big(\Sigma_2^{\delta_2}\big)^{\rm{c}}.$}

\medskip
Since on $\big(\Sigma_2^{\delta_2}\big)^{\rm{c}}$ we have
$$|V_3|^2> \delta_2((\lambda_1+\lambda_2)^2+(\lambda_2+\lambda_3)^2+
(\lambda_3+\lambda_1)^2)|V_2|^2,$$
it follows that
\[
|B_1V_1+B_2V_2+B_3V_3|^2\prec |V_3|^2.
\]
Then, for $V\in \big(\Sigma_1^{\delta_1}\big)^{\rm{c}}\cap\big(\Sigma_2^{\delta_2}\big)^{\rm{c}}$,
for suitable $\gamma_1, \gamma_2, \gamma_3$ (independent of $V$),
\begin{multline*}
|(\lambda_2\lambda_3)V_1-(\lambda_2+\lambda_3)V_2+V_3|^2+
|(\lambda_3\lambda_1)V_1-(\lambda_3+\lambda_1)V_2+V_3|^2 +
|(\lambda_1\lambda_2)V_1-(\lambda_1+\lambda_2)V_2+V_3|^2\\
\succ\gamma_3|V_3|^2-\gamma_2((\lambda_1+\lambda_2)^2+(\lambda_2+\lambda_3)^2+(\lambda_3+\lambda_1)^2)|V_2|^2\\
-\gamma_1\frac{1}{\delta_1}(|V_3|^2+((\lambda_1+\lambda_2)^2+(\lambda_2+\lambda_3)^2+(\lambda_3+\lambda_1)^2)|V_2|^2)\\
\succ (\gamma_3-\gamma_1\frac{1}{\delta_1})|V_3|^2-(\gamma_2+\gamma_1\frac{1}{\delta})\frac{1}{\delta_2}|V_3|^2.
\end{multline*}
We conclude that for $\delta_1$ and $\delta_2$ big enough,
\begin{multline*}
|(\lambda_2\lambda_3)V_1-(\lambda_2+\lambda_3)V_2+V_3|^2+
|(\lambda_3\lambda_1)V_1-(\lambda_3+\lambda_1)V_2+V_3|^2\\
+|(\lambda_1\lambda_2)V_1-(\lambda_1+\lambda_2)V_2+V_3|^2\succ |V_3|^2,
\end{multline*}
and, therefore, \eqref{m3} holds in the area  $\big(\Sigma_1^{\delta_1}\big)^{\rm{c}}\cap\big(\Sigma_2^{\delta_2}\big)^{\rm{c}}$.\\

The next table describes and summarises the proof above:\\
\[
\begin{tabular}{|c|c|c|}
  \hline
  {\bf Area} & {\bf Estimates in}  &  {\bf $\delta_i$}\\
  \hline
  $\Sigma_1^{\delta_1}$ & $|V_1|^2$ & any $\delta_1$ \\
  \hline
  $\big(\Sigma_1^{\delta_1}\big)^{\rm{c}}\cap\Sigma_2^{\delta_2}$ & $|V_2|^2$ & $\delta_1$ big, any $\delta_2$ \\
  \hline
  $\big(\Sigma_1^{\delta_1}\big)^{\rm{c}}\cap\big(\Sigma_2^{\delta_2}\big)^{\rm{c}}$ & $|V_3|^2$ & $\delta_1$ and $\delta_2$ big \\
  \hline
\end{tabular}
\]
\vspace{0.3cm}
\subsection{The general case $m$}
Inspired by the previous subsection we now deal with the inequality
\[
|\mathcal{W}BV|\prec |\mathcal{W}V|
\]
for $\mathcal{W}=\mathcal{W}^{(m)}$ and arbitrary $m\in\N$. This is the topic of the following theorem where the coefficients $\sigma^{(m)}_h(\lambda)$ are defined as in Section \ref{SEC:qs} and $\lambda=(\lambda_1,\lambda_2,...,\lambda_m)\in\R^m$ is the vector of the eigenvalues of the matrix $A(t,\xi)$ (or the 0-order normalised roots) satisfying the condition \eqref{LC}.
\begin{thm}
\label{theo_LC}
Let the entries $B_j$ of the matrix
\[
B=\left(
    \begin{array}{ccccc}
      0 & 0 & 0 & \dots & 0\\
      0 & 0 & 0& \dots & 0 \\
      \dots & \dots & \dots & \dots & 0 \\
      B_1 & B_2 & \dots & \dots & B_m \\
    \end{array}
  \right)
\]
in \eqref{system_new} fulfil the condition
\beq
\label{LB}
|B_j|^2\prec \sum_{i=1}^m |\sigma_{m-j}^{(m-1)}(\pi_i\lambda)|^2
\eeq
for $j=1,...,m$. Then we have
\[
|\mathcal{W}BV|\prec |\mathcal{W}V|
\]
uniformly over all $V\in \C^m$. More precisely, define
\[
\Sigma_k^{\delta_k}:=\{V\in\C^m:\, |V_m|^2+\sum_{j=k+1}^{m-1}\sum_{i=1}^m|\sigma_{m-j}^{(m-1)}(\pi_i\lambda)|^2|V_j|^2
\le\delta_k\sum_{i=1}^m|\sigma_{m-k}^{(m-1)}(\pi_i\lambda)|^2|V_k|^2\},
\]
for $k=1,...,m-1$. Then, there exist suitable $\delta_j>0$, $j=1,...,m-1,$ such that
\[
\begin{split}
|\mathcal{W}BV|^2&\prec \sum_{i=1}^m |\sigma_{m-1}^{(m-1)}(\pi_i\lambda)|^2 |V_1|^2,
\\
|\mathcal{W}V|^2&\succ \sum_{i=1}^m |\sigma_{m-1}^{(m-1)}(\pi_i\lambda)|^2 |V_1|^2
\end{split}
\]
on $\Sigma_1^{\delta_1}$,
\[
\begin{split}
|\mathcal{W}BV|^2&\prec \sum_{i=1}^m |\sigma_{m-k}^{(m-1)}(\pi_i\lambda)|^2 |V_k|^2,
\\
|\mathcal{W}V|^2&\succ \sum_{i=1}^m |\sigma_{m-k}^{(m-1)}(\pi_i\lambda)|^2 |V_k|^2
\end{split}
\]
on $$\big(\Sigma_1^{\delta_1}\big)^{\rm{c}}\cap\big(\Sigma_2^{\delta_2}\big)^{\rm{c}}\cap\cdots\cap
\big(\Sigma_{k-1}^{\delta_{k-1}}\big)^{\rm{c}}\cap\Sigma_k^{\delta_k}$$ for $2\le k\le m-1$, and
\[
\begin{split}
|\mathcal{W}BV|^2&\prec \sum_{i=1}^m |\sigma_{0}^{(m-1)}(\pi_i\lambda)|^2 |V_m|^2,
\\
|\mathcal{W}V|^2&\succ \sum_{i=1}^m |\sigma_{0}^{(m-1)}(\pi_i\lambda)|^2 |V_m|^2
\end{split}
\]
on $\big(\Sigma_1^{\delta_1}\big)^{\rm{c}}\cap\big(\Sigma_2^{\delta_2}\big)^{\rm{c}}\cap\cdots\cap
\big(\Sigma_{m-1}^{\delta_{m-1}}\big)^{\rm{c}}$.
\end{thm}
Note that \eqref{LB} is a reformulation of the condition \eqref{EQ:lot2} on the lower order terms. The proof of Theorem \ref{theo_LC} makes use of the following two lemmas.
\begin{lem}
\label{lemma1}
For all $i$ and $j$ with $1\le i,j\le m$ and $k=1,...,m-1,$ one has
\begin{multline}
\label{formula_diff}
\sigma_{m-k}^{(m-1)}(\pi_i\lambda)-\sigma_{m-k}^{(m-1)}(\pi_j\lambda)\\
=(-1)^{m-k}(\lambda_j-\lambda_i)
\sum_{\substack{i_h\neq i,\, i_h\neq j\\ 1\le i_1<i_2<\cdots<i_{m-k-1}\le m}} \lambda_{i_1}\lambda_{i_2}\cdots\lambda_{i_{m-k-1}}.
\end{multline}
\end{lem}
\begin{proof}
By definition of $\sigma_{m-k}^{(m-1)}(\pi_i\lambda)$ and $\sigma_{m-k}^{(m-1)}(\pi_j\lambda)$ we have that
\begin{multline*}
\sigma_{m-k}^{(m-1)}(\pi_i\lambda)=(-1)^{m-k}\sum_{\substack{1\le l_1<l_2<\cdots<l_{m-k}\le m\\ l_h\neq i}}\lambda_{l_1}\lambda_{l_2}\cdots\lambda_{l_{m-k}}\\
= (-1)^{m-k}\sum_{\substack{1\le l_1<l_2<\cdots<l_{m-k}\le m\\ l_h\neq i,j}}\lambda_{l_1}\lambda_{l_2}\cdots\lambda_{l_{m-k}}\\
+(-1)^{m-k}\lambda_j\sum_{\substack{1\le l_1<l_2<\cdots<l_{m-k-1}\le m\\ l_h\neq i,j}}\lambda_{l_1}\lambda_{l_2}\cdots\lambda_{l_{m-k-1}}
\end{multline*}
and
\begin{multline*}
\sigma_{m-k}^{(m-1)}(\pi_j\lambda)=(-1)^{m-k}\sum_{\substack{1\le l_1<l_2<\cdots<l_{m-k}\le m\\ l_h\neq j}}\lambda_{l_1}\lambda_{l_2}\cdots\lambda_{l_{m-k}}\\
= (-1)^{m-k}\sum_{\substack{1\le l_1<l_2<\cdots<l_{m-k}\le m\\ l_h\neq i,j}}\lambda_{l_1}\lambda_{l_2}\cdots\lambda_{l_{m-k}}\\
+(-1)^{m-k}\lambda_i\sum_{\substack{1\le l_1<l_2<\cdots<l_{m-k-1}\le m\\ l_h\neq i,j}}\lambda_{l_1}\lambda_{l_2}\cdots\lambda_{l_{m-k-1}}.
\end{multline*}
This leads immediately to the formula \eqref{formula_diff}.
\end{proof}
\begin{lem}
\label{lemma2}
For all $k=1,...,m$, we have
\beq
\label{formula_k}
\sum_{i=1}^m\biggl|\sum_{j=k+1}^m \sigma_{m-j}^{(m-1)}(\pi_i\lambda)V_j+\sigma_{m-k}^{(m-1)}(\pi_i\lambda)V_k\biggr|^2\succ
\sum_{i=1}^m|\sigma_{m-k}^{(m-1)}(\pi_i\lambda)|^2|V_k|^2.
\eeq
\end{lem}
\begin{proof}
We give a proof by induction on the order $m$. Setting $m=2$ the estimates above makes sense for $k=1$. Hence we have to prove that
\[
\sum_{i=1}^2|\sigma_{0}^{(1)}(\pi_i\lambda)V_2+\sigma_{1}^{(1)}(\pi_i\lambda)V_1|^2= \sum_{i=1}^2|V_2+\sigma_{1}^{(1)}(\pi_i\lambda)V_1|^2 \succ
\sum_{i=1}^2|\sigma_{1}^{(1)}(\pi_i\lambda)|^2|V_1|^2.
\]
This is clear since by the condition \eqref{LC} we have that
\begin{multline*}
\sum_{i=1}^2|V_2+\sigma_{1}^{(1)}(\pi_i\lambda)V_1|^2 \succ |\sigma^{(1)}_1(\pi_1\lambda)-\sigma^{(1)}_1(\pi_2\lambda)|^2|V_1|^2=(\lambda_2-\lambda_1)^2|V_1|^2\\
\succ (\lambda_1^2+\lambda_2^2)|V_1|^2=\sum_{i=1}^2|\sigma_{1}^{(1)}(\pi_i\lambda)|^2|V_1|^2.
\end{multline*}
Assume now that \eqref{formula_k} holds for $m-1$. Estimating the left-hand side of \eqref{formula_k} with the differences between two arbitrary summands we can write
\begin{multline*}
\sum_{i=1}^m\biggl|\sum_{j=k+1}^m \sigma_{m-j}^{(m-1)}(\pi_i\lambda)V_j+\sigma_{m-k}^{(m-1)}(\pi_i\lambda)V_k\biggr|^2\\
\succ \sum_{1\le l_1\neq l_2\le m}\biggl|\sum_{j=k+1}^m (\sigma_{m-j}^{(m-1)}(\pi_{l_1}\lambda)-\sigma_{m-j}^{(m-1)}(\pi_{l_2}\lambda))V_j+ (\sigma_{m-k}^{(m-1)}(\pi_{l_1}\lambda)-\sigma_{m-k}^{(m-1)}(\pi_{l_2}\lambda))V_k\biggr|^2\\
=\sum_{1\le l_1\neq l_2\le m}\biggl|\sum_{j=k+1}^{m-1} (\sigma_{m-j}^{(m-1)}(\pi_{l_1}\lambda)-\sigma_{m-j}^{(m-1)}(\pi_{l_2}\lambda))V_j+ (\sigma_{m-k}^{(m-1)}(\pi_{l_1}\lambda)-\sigma_{m-k}^{(m-1)}(\pi_{l_2}\lambda))V_k\biggr|^2.
\end{multline*}
By applying Lemma \ref{lemma1} and the condition \eqref{LC} we obtain the following bound from below:
\begin{multline*}
\sum_{i=1}^m\biggl|\sum_{j=k+1}^m \sigma_{m-j}^{(m-1)}(\pi_i\lambda)V_j+\sigma_{m-k}^{(m-1)}(\pi_i\lambda)V_k\biggr|^2\succ \sum_{1\le l_1\neq l_2\le m}\biggl|-(\lambda_{l_2}-\lambda_{l_1})V_{m-1}+\\
\sum_{j=k+1}^{m-2}(-1)^{m-j}(\lambda_{l_2}-\lambda_{l_1})\sum_{\substack{i_h\neq l_1,\, i_h\neq l_2\\ 1\le i_1<i_2<\cdots<i_{m-j-1}\le m}} (\lambda_{i_1}\lambda_{i_2}\cdots\lambda_{i_{m-j-1}})V_j\\
+(-1)^{m-k}(\lambda_{l_2}-\lambda_{l_1})
\sum_{\substack{i_h\neq l_1,\, i_h\neq l_2\\ 1\le i_1<i_2<\cdots<i_{m-k-1}\le m}} (\lambda_{i_1}\lambda_{i_2}\cdots\lambda_{i_{m-k-1}})V_k\biggr|^2\\
\succ \sum_{1\le l_1\neq l_2\le m}(\lambda_{l_1}^2+\lambda_{l_2}^2)\biggl|-V_{m-1}+\sum_{j=k+1}^{m-2}(-1)^{m-j}\sum_{\substack{i_h\neq l_1,\, i_h\neq l_2\\ 1\le i_1<i_2<\cdots<i_{m-j-1}\le m}} (\lambda_{i_1}\lambda_{i_2}\cdots\lambda_{i_{m-j-1}})V_j\\
+(-1)^{m-k}\sum_{\substack{i_h\neq l_1,\, i_h\neq l_2\\ 1\le i_1<i_2<\cdots<i_{m-k-1}\le m}} (\lambda_{i_1}\lambda_{i_2}\cdots\lambda_{i_{m-k-1}})V_k\biggr|^2.
\end{multline*}
Noting that
\[
(-1)^{m-j}\sum_{\substack{i_h\neq l_1,\, i_h\neq l_2\\ 1\le i_1<i_2<\cdots<i_{m-j-1}\le m}} (\lambda_{i_1}\lambda_{i_2}\cdots\lambda_{i_{m-j-1}})= -\sigma^{(m-2)}_{m-1-j}(\pi_{l_2}(\pi_{l_1}\lambda))
\]
for $j=k,...,m-2$, we write the estimate above as
\begin{multline}
\label{induc_est}
\sum_{i=1}^m\biggl|\sum_{j=k+1}^m \sigma_{m-j}^{(m-1)}(\pi_i\lambda)V_j+\sigma_{m-k}^{(m-1)}(\pi_i\lambda)V_k\biggr|^2\\
\succ\sum_{1\le l_1\neq l_2\le m}(\lambda_{l_1}^2+\lambda_{l_2}^2)\biggl|-V_{m-1}-\sum_{j=k+1}^{m-2}\sigma^{(m-2)}_{m-1-j}(\pi_{l_2}(\pi_{l_1}\lambda)) V_j-\sigma^{(m-2)}_{m-1-k}(\pi_{l_2}(\pi_{l_1}\lambda))V_k\biggr|^2\\
=\sum_{1\le l_1\neq l_2\le m}(\lambda_{l_1}^2+\lambda_{l_2}^2)\biggl|V_{m-1}+\sum_{j=k+1}^{m-2}\sigma^{(m-2)}_{m-1-j}(\pi_{l_2}(\pi_{l_1}\lambda)) V_j+\sigma^{(m-2)}_{m-1-k}(\pi_{l_2}(\pi_{l_1}\lambda))V_k\biggr|^2\\
=\sum_{l_1}\lambda_{l_1}^2\sum_{1\le l_1\neq l_2\le m}\biggl|V_{m-1}+\sum_{j=k+1}^{m-2}\sigma^{(m-2)}_{m-1-j}(\pi_{l_2}(\pi_{l_1}\lambda)) V_j+\sigma^{(m-2)}_{m-1-k}(\pi_{l_2}(\pi_{l_1}\lambda))V_k\biggr|^2\\
+\sum_{l_2}\lambda_{l_2}^2\sum_{1\le l_1\neq l_2\le m}\biggl|V_{m-1}+\sum_{j=k+1}^{m-2}\sigma^{(m-2)}_{m-1-j}(\pi_{l_2}(\pi_{l_1}\lambda)) V_j+\sigma^{(m-2)}_{m-1-k}(\pi_{l_2}(\pi_{l_1}\lambda))V_k\biggr|^2.
\end{multline}
By now applying the inductive hypothesis to the last two summands in \eqref{induc_est} we obtain
\begin{multline*}
\sum_{i=1}^m\biggl|\sum_{j=k+1}^m \sigma_{m-j}^{(m-1)}(\pi_i\lambda)V_j+\sigma_{m-k}^{(m-1)}(\pi_i\lambda)V_k\biggr|^2\\
\succ\sum_{l_1}\lambda_{l_1}^2\sum_{1\le l_1\neq l_2\le m}| \sigma^{(m-2)}_{m-1-k}(\pi_{l_2}(\pi_{l_1}\lambda))|^2|V_k|^2
+\sum_{l_2}\lambda_{l_2}^2\sum_{1\le l_1\neq l_2\le m}| \sigma^{(m-2)}_{m-1-k}(\pi_{l_1}(\pi_{l_2}\lambda))|^2|V_k|^2\\
\succ \sum_{i=1}^m |\sigma_{m-k}^{(m-1)}(\pi_i\lambda)|^2 |V_k|^2,
\end{multline*}
which completes the proof.
\end{proof}
\begin{proof}[Proof of Theorem \ref{theo_LC}]
By definition of the matrices $\mathcal{W}$ and $B$ we have that $|\mathcal{W}BV|^2\prec |\mathcal{W}V|^2$ is equivalent to
\beq
\label{eq_est}
\biggl|\sum_{j=1}^m B_jV_j\biggr|^2\prec\sum_{i=1}^m\biggl|\sum_{j=1}^m\sigma^{(m-1)}_{m-j}(\pi_i\lambda)V_j\biggr|^2.
\eeq
Making use of the conditions \eqref{LB} we have that the following estimate is valid on the area $\Sigma_1^{\delta_1}$:
\begin{multline*}
\biggl|\sum_{j=1}^m B_jV_j\biggr|^2\prec \sum_{j=1}^m |B_j|^2|V_j|^2\prec\sum_{j=1}^m\sum_{i=1}^m |\sigma_{m-j}^{(m-1)}(\pi_i\lambda)|^2|V_j|^2\\
\prec |V_m|^2+\sum_{j=2}^{m-1}\sum_{i=1}^m|\sigma^{(m-1)}_{m-j}(\pi_i\lambda)|^2|V_j|^2
+\sum_{i=1}^m|\sigma^{(m-1)}_{m-1}(\pi_i\lambda)|^2|V_1|^2\\
\prec (1+\delta_1)\sum_{i=1}^m|\sigma^{(m-1)}_{m-1}(\pi_i\lambda)|^2|V_1|^2\prec \sum_{i=1}^m|\sigma^{(m-1)}_{m-1}(\pi_i\lambda)|^2|V_1|^2.
\end{multline*}
Setting $k=1$ in Lemma \ref{lemma2} we obtain the bound from below
\begin{multline*}
\sum_{i=1}^m\biggl|\sum_{j=1}^m\sigma^{(m-1)}_{m-j}(\pi_i\lambda)V_j\biggr|^2= \sum_{i=1}^m\biggl|\sum_{j=2}^m\sigma^{(m-1)}_{m-j}(\pi_i\lambda)V_j+\sigma^{(m-1)}_{m-1}(\pi_i\lambda)V_1\biggr|^2\\
\succ \sum_{i=1}^m|\sigma^{(m-1)}_{m-1}(\pi_i\lambda)|^2|V_1|^2.
\end{multline*}
This proves the inequality \eqref{eq_est} on $\Sigma_{1}^{\delta_1}$ for any $\delta_1>0$.

Let us now assume that $V\in\big(\Sigma_1^{\delta_1}\big)^{\rm{c}}\cap\big(\Sigma_2^{\delta_2}\big)^{\rm{c}}\cap\cdots\cap
\big(\Sigma_{k-1}^{\delta_{k-1}}\big)^{\rm{c}}\cap\Sigma_k^{\delta_k}$ for $2\le k\le m-1$. By definition of the regions $\Sigma_h^{\delta_h}$ and taking $\delta_h\ge 1$ for $1\le h\le k-1$ we have that
\begin{multline*}
\sum_{i=1}^m |\sigma^{(m-1)}_{m-(k-1)}(\pi_i\lambda)|^2|V_{k-1}|^2 <\frac{1}{\delta_{k-1}}\biggl(|V_m|^2+\sum_{j=k+1}^{m-1}\sum_{i=1}^m|\sigma^{(m-1)}_{m-j}(\pi_i\lambda)|^2|V_j|^2\\
+\sum_{i=1}^m|\sigma^{(m-1)}_{m-k}(\pi_i\lambda)|^2|V_k|^2\biggl) \le \frac{1}{\delta_{k-1}}(1+\delta_k)\sum_{i=1}^m|\sigma^{(m-1)}_{m-k}(\pi_i\lambda)|^2|V_k|^2,
\end{multline*}
\begin{multline*}
\sum_{i=1}^m |\sigma^{(m-1)}_{m-(k-2)}(\pi_i\lambda)|^2|V_{k-2}|^2 <\frac{1}{\delta_{k-2}}\biggl(|V_m|^2+\sum_{j=k+1}^{m-1}\sum_{i=1}^m|\sigma^{(m-1)}_{m-j}(\pi_i\lambda)|^2|V_j|^2\\
+\sum_{i=1}^m|\sigma^{(m-1)}_{m-k}(\pi_i\lambda)|^2|V_k|^2+
\sum_{i=1}^m|\sigma^{(m-1)}_{m-(k-1)}(\pi_i\lambda)|^2|V_{k-1}|^2 \biggl)\\
\le\frac{1}{\delta_{k-2}}\big(1+\delta_k+\frac{1}{\delta_{k-1}}(1+\delta_k)\big)
\sum_{i=1}^m|\sigma^{(m-1)}_{m-k}(\pi_i\lambda)|^2|V_k|^2,\\
\le (1+\delta_k)\big(\frac{1}{\delta_{k-1}}+\frac{1}{\delta_{k-2}})\sum_{i=1}^m|\sigma^{(m-1)}_{m-k}(\pi_i\lambda)|^2|V_k|^2.
\end{multline*}
By iteration one can easily prove the following bound
\beq
\label{bound}
\sum_{i=1}^m |\sigma^{(m-1)}_{m-j}(\pi_i\lambda)|^2|V_{j}|^2\le (1+\delta_k)\sum_{h=1}^{k-1}\frac{1}{\delta_h}\sum_{i=1}^m|\sigma^{(m-1)}_{m-k}(\pi_i\lambda)|^2|V_k|^2,
\eeq
valid on the region $\big(\Sigma_1^{\delta_1}\big)^{\rm{c}}\cap\big(\Sigma_2^{\delta_2}\big)^{\rm{c}}\cap\cdots\cap
\big(\Sigma_{k-1}^{\delta_{k-1}}\big)^{\rm{c}}\cap\Sigma_k^{\delta_k}$ for all $j$ with $1\le j\le k-1$.

It follows that
\begin{multline*}
\biggl|\sum_{j=1}^m B_jV_j\biggr|^2\prec \sum_{j=1}^m |B_j|^2|V_j|^2\prec\sum_{j=1}^m\sum_{i=1}^m |\sigma_{m-j}^{(m-1)}(\pi_i\lambda)|^2|V_j|^2\\
\prec\sum_{j=k+1}^m\sum_{i=1}^m |\sigma_{m-j}^{(m-1)}(\pi_i\lambda)|^2|V_j|^2+\sum_{i=1}^m |\sigma_{m-k}^{(m-1)}(\pi_i\lambda)|^2|V_k|^2+\sum_{j=1}^{k-1}\sum_{i=1}^m |\sigma_{m-j}^{(m-1)}(\pi_i\lambda)|^2|V_j|^2\\
\prec \sum_{i=1}^m |\sigma_{m-k}^{(m-1)}(\pi_i\lambda)|^2|V_k|^2.
\end{multline*}
We now pass to estimate the right-hand side of \eqref{eq_est} making use of Lemma \ref{lemma2} and of the bound \eqref{bound}. We obtain
\begin{multline*}
\sum_{i=1}^m\biggl|\sum_{j=1}^m\sigma^{(m-1)}_{m-j}(\pi_i\lambda)V_j\biggr|^2\\
\succ \sum_{i=1}^m\gamma_1\biggl|\sum_{j=k+1}^{m}\sigma^{(m-1)}_{m-j}(\pi_i\lambda)V_j
+\sigma^{(m-1)}_{m-k}(\pi_i\lambda)V_k\biggr|^2-
\gamma_2\sum_{i=1}^m\sum_{j=1}^{k-1}|\sigma^{(m-1)}_{m-j}(\pi_i\lambda)|^2|V_j|^2\\
\succ \gamma_1\sum_{i=1}^m|\sigma_{m-k}^{(m-1)}(\pi_i\lambda)|^2|V_k|^2-\gamma_2
(1+\delta_k)\sum_{h=1}^{k-1}\frac{1}{\delta_h}\sum_{i=1}^m|\sigma^{(m-1)}_{m-k}(\pi_i\lambda)|^2|V_k|^2\\
=\biggl(\gamma_1-\gamma_2(1+\delta_k)\sum_{h=1}^{k-1}\frac{1}{\delta_h}\biggr)
\sum_{i=1}^m|\sigma^{(m-1)}_{m-k}(\pi_i\lambda)|^2|V_k|^2.
\end{multline*}
Therefore, the estimate \eqref{eq_est} holds in the region $\big(\Sigma_1^{\delta_1}\big)^{\rm{c}}\cap\big(\Sigma_2^{\delta_2}\big)^{\rm{c}}\cap\cdots\cap
\big(\Sigma_{k-1}^{\delta_{k-1}}\big)^{\rm{c}}\cap\Sigma_k^{\delta_k}$ for any $\delta_k>0$ choosing $\delta_1, \delta_2,...,\delta_{k-1}$ big enough.

We conclude the proof by assuming $V\in\big(\Sigma_1^{\delta_1}\big)^{\rm{c}}\cap\big(\Sigma_2^{\delta_2}\big)^{\rm{c}}\cap\cdots\cap
\big(\Sigma_{m-1}^{\delta_{m-1}}\big)^{\rm{c}}$. Since
\[
|V_m|^2+\sum_{j=h+1}^{m-1}\sum_{i=1}^m|\sigma^{(m-1)}_{m-j}(\pi_i\lambda)|^2|V_j|^2>
\delta_h\sum_{i=1}^m|\sigma^{(m-1)}_{m-h}(\pi_i\lambda)|^2|V_h|^2,
\]
for $1\le h\le m-1$, arguing as above and taking $\delta_h\ge 1$ we obtain the estimate
\beq
\label{bound2}
\sum_{i=1}^m |\sigma^{(m-1)}_{m-j}(\pi_i\lambda)|^2|V_{j}|^2\le \sum_{h=1}^{m-1}\frac{1}{\delta_h}|V_m|^2
\eeq
valid on the region $\big(\Sigma_1^{\delta_1}\big)^{\rm{c}}\cap\big(\Sigma_2^{\delta_2}\big)^{\rm{c}}\cap\cdots\cap
\big(\Sigma_{m-1}^{\delta_{m-1}}\big)^{\rm{c}}$ for all $j$ with $1\le j\le m-1$.
Hence,
\[
\biggl|\sum_{j=1}^m B_jV_j\biggr|^2\prec \sum_{j=1}^m |B_j|^2|V_j|^2\prec\sum_{j=1}^m\sum_{i=1}^m |\sigma_{m-j}^{(m-1)}(\pi_i\lambda)|^2|V_j|^2\prec |V_m|^2
\]
and
\begin{multline*}
\sum_{i=1}^m\biggl|\sum_{j=1}^m\sigma^{(m-1)}_{m-j}(\pi_i\lambda)V_j\biggr|^2\succ \gamma_1|V_m|^2-\gamma_2\sum_{i=1}^m\biggl|\sum_{j=1}^{m-1}
\sigma^{(m-1)}_{m-j}(\pi_i\lambda)V_j\biggr|^2\\
\succ \gamma_1|V_m|^2-\gamma_2\sum_{j=1}^{m-1}\sum_{i=1}^m|\sigma^{(m-1)}_{m-j}(\pi_i\lambda)|^2|V_j|^2\\
\succ \biggl(\gamma_1-\gamma_2(m-1)\sum_{h=1}^{m-1}\frac{1}{\delta_h}\biggr)|V_m|^2.
\end{multline*}
This means that the inequality \eqref{eq_est} holds on $\big(\Sigma_1^{\delta_1}\big)^{\rm{c}}\cap\big(\Sigma_2^{\delta_2}\big)^{\rm{c}}\cap\cdots\cap
\big(\Sigma_{m-1}^{\delta_{m-1}}\big)^{\rm{c}}$ for sufficiently large values of $\delta_1,\delta_2,...,\delta_{m-1}$.
\end{proof}

\section{Well-posedness results}
\label{SEC:wp}

We are now ready to prove the well-posedness results
given in Theorem \ref{theo_GRi}. For the advantage of the reader and the sake of simplicity we reformulate Theorem \ref{theo_GRi} as the following Theorem \ref{theo_GR} where we make use of the language and notations introduced in Theorem \ref{theo_LC}.
\begin{thm}
\label{theo_GR}
Assume $A_j\in{C}([0,T])$ for all $j$.
If the coefficients $A_{(j)}\in C^\infty([0,T])$, the characteristic roots are real and satisfy \eqref{LC} and the entries of the matrix $B$ of the lower order terms in \eqref{system_new} fulfill the conditions \eqref{LB} for $\xi$ away from $0$ then the Cauchy problem \eqref{CP} is well-posed in any Gevrey space. More precisely,
\begin{itemize}
\item[(i)] if $A_{(j)}\in{C}^k([0,T])$ for some $k\ge 2$ and $g_j\in G^s(\R^n)$ for
$j=1,...m,$ then there exists a unique solution $u\in C^m([0,T];G^s(\R^n))$ for
\[
1\le s<1+\frac{k}{2(m-1)};
\]
\item[(ii)] if $A_{(j)}\in{C}^k([0,T])$ for some $k\ge 2$ and $g_j\in \mathcal{E}'_{(s)}(\R^n)$ for $j=1,...m,$ then there exists a unique solution $u\in C^m([0,T];\mathcal{D}'_{(s)}(\R^n))$ for
\[
1\le s\le 1+\frac{k}{2(m-1)}.
\]
\end{itemize}
\end{thm}
\begin{proof}
As usual, the well-posedness in the case of $s=1$ follows from the result of
Bony and Shapira, so we may assume $s>1$. By finite propagation speed for hyperbolic equations it is not restrictive to take compactly supported initial data and therefore to have that the solution $u$ is compactly supported in $x$.

Combining the energy estimate \eqref{EE} with the estimates of the first, second and third term in Section \ref{SEC:se} we obtain the estimate
\[
\partial_t E_\eps(t,\xi)\le (C_1\eps^{-2(m-1)/k}+C_2\eps\lara{\xi}+C_3)E_\eps(t,\xi),
\]
valid for some positive constants $C_1, C_2, C_3$, for $t\in[0,T]$ and $|\xi|\ge R$.
Here, the estimate for the third term is provided by Theorem \ref{theo_LC} for $|\xi|\ge R$.
A straightforward application of Gronwall's lemma leads to
\[
E_\eps(t,\xi)\le E_\eps(0,\xi)\esp^{C_1T\eps^{-2(m-1)/k}+C_2T\eps\lara{\xi}+C_3T}\le E_\eps(0,\xi)C_T\esp^{C_T(\eps^{-2(m-1)/k}+\eps\lara{\xi})}.
\]
Setting $\eps^{-2(m-1)/k}=\eps\lara{\xi}$ we get
\[
E_\eps(t,\xi)\le E_\eps(0,\xi)C_T\esp^{C_T\lara{\xi}^{\frac{1}{\sigma}}},
\]
where $\sigma=1+k/[2(m-1)]$. Finally, making use of the inequality \eqref{p2} we arrive at
\[
C_m^{-1}\eps^{2(m-1)}|V(t,\xi)|^2\le E_\eps(t,\xi)\le E_\eps(0,\xi)C_T\esp^{C_T\lara{\xi}^{\frac{1}{\sigma}}}\le C_m|V(0,\xi)|^2C_T\esp^{C_T\lara{\xi}^{\frac{1}{\sigma}}},
\]
which implies
\beq
\label{est_V}
|V(t,\xi)|\le C\lara{\xi}^{\frac{k}{2\sigma}}\esp^{C\lara{\xi}^{\frac{1}{\sigma}}}|V(0,\xi)|,
\eeq
for some new constant $C>0$, for $t\in[0,T]$ and $|\xi|\ge R$.\\

(i) Recall that $V(t,\xi)=\mathcal{F}_{x\to\xi}U(t,x)$, where $U$ is the $u_j$'s column vector. If the initial data $g_l$ belong to $G^s_0(\R^n)$ from the Fourier transform characterisation of Gevrey functions (\cite[Proposition 2.2]{GR:11}) we have that $|V(0,\xi)|\le c\,\esp^{-\delta\lara{\xi}^{\frac{1}{s}}}$ for some constants $c>0$ and $\delta>0$. Hence,
\[
|V(t,\xi)|\le C\lara{\xi}^{\frac{k}{2\sigma}}\esp^{C\lara{\xi}^{\frac{1}{\sigma}}}c\,\esp^{-\delta\lara{\xi}^{\frac{1}{s}}}
\]
for all $t\in[0,T]$ and $\xi\in\R^n$. Let $s<\sigma$. Then
$V(t,\xi)$ defines a tempered distribution in $\S'(\R^n)$ such that
\[
|V(t,\xi)|\le Cc\lara{\xi}^{\frac{k}{2\sigma}}\esp^{C\lara{\xi}^{\frac{1}{\sigma}}}
\esp^{-\frac{\delta}{2}\lara{\xi}^{\frac{1}{s}}}\esp^{-\frac{\delta}{2}\lara{\xi}^{\frac{1}{s}}}\\
\le Cc\lara{\xi}^{\frac{k}{2\sigma}}\esp^{\lara{\xi}^{\frac{1}{\sigma}}
(C-\frac{\delta}{2}\lara{\xi}^{\frac{1}{s}-\frac{1}{\sigma}})}
\esp^{-\frac{\delta}{2}\lara{\xi}^{\frac{1}{s}}}.
\]
It follows that
\begin{equation}\label{estv1}
|V(t,\xi)|\le c'\esp^{-\frac{\delta}{2}\lara{\xi}^{\frac{1}{s}}},
\end{equation}
for some $c',\delta>0$ and for $|\xi|$ large enough. This is sufficient to prove that $U(t,x)$ belongs to the Gevrey class $G^s(\R^n)$ for all $t\in[0,T]$ and that the Cauchy problem \eqref{CP} has a unique solution $u\in C^m([0,T];G^s(\R^n))$ for $s<\sigma$ under the assumptions of case (i).\\

(ii) If the initial data $g_l$ are Gevrey Beurling ultradistributions in $\mathcal{E}'_{(s)}(\R^n)$, from the Fourier transform characterisation of ultradistributions (\cite[Proposition 2.13]{GR:11}) we have that there exist $\delta>0$ and $c>0$ such that $|V(0,\xi)|\le c\,\esp^{\delta\lara{\xi}^{\frac{1}{s}}}$ for all $\xi\in\R^n$.
Hence, taking $s\le\sigma$, we obtain the estimate
\[
|V(t,\xi)|\le Cc\lara{\xi}^{\frac{k}{2\sigma}}\esp^{C\lara{\xi}^{\frac{1}{\sigma}}}
\esp^{+\delta\lara{\xi}^{\frac{1}{s}}}\le c'\esp^{+\delta'\lara{\xi}^{\frac{1}{s}}}
\]
for some $c',\delta'>0$. This proves that the Cauchy problem \eqref{CP} has a unique solution $u\in C^m([0,T];\mathcal{D}'_{(s)}(\R^n))$ for $s\le\sigma$ under the assumptions of case (ii).
\end{proof}
We pass to consider the case of analytic coefficients. We prove $C^\infty$ and distributional well-posedness of the Cauchy problem \eqref{CP} providing an extension of Theorem 1 in \cite{KS} to any space dimension. Our proof makes use of the following lemma on analytic functions,
a parameter-dependent version of the statement (61)-(62) in \cite{KS}.
\begin{lem}
\label{lem_par}
Let $f(t,\xi)$ be an analytic function in $t\in[0,T]$, continuous and homogeneous of order $0$ in $\xi\in\R^n$. Then,
\begin{itemize}
\item[(i)] for all $\xi$ there exists a partition $(\tau_{h(\xi)})$ of the interval $[0,T]$ such that
\[
0=\tau_0<\tau_1<\cdots<\tau_{h(\xi)}<\cdots<\tau_{N(\xi)}=T
\]
with $\sup_{\xi\neq 0}N(\xi)<\infty$, such that
\item[(ii)] there exists $C>0$ such that
\[
|\partial_t f(t,\xi)|\le C\biggl(\frac{1}{t-\tau_{h(\xi)}}+\frac{1}{\tau_{{h+1}(\xi)}-t}\biggr)|f(t,\xi)|
\]
for all $t\in(\tau_{h(\xi)},\tau_{{h+1}(\xi)})$, $\xi\in\R^n$ with $\xi\neq 0$ and $0\le h(\xi)\le N(\xi)$.
\end{itemize}
\end{lem}
\begin{proof}
Since the function $f$ is homogeneous of order $0$ in $\xi$ we can assume $|\xi|=1$. Excluding the trivial case $f\equiv 0$ we have that $f(t,\xi)$ has a finite number of zeroes in $[0,T]$ and hence we can find a partition $(\tau_{h(\xi)})$ as in (i) such that $f(t,\xi)\neq 0$ in each interval $(\tau_{h(\xi)},\tau_{h+1(\xi)})$, taking $\tau_{h(\xi)}$,
$1\leq h(\xi)\leq N(\xi)-1$, to be the zeros of $f(\cdot,\xi)$.

Note that the function $N(\xi)$ is locally bounded and, therefore, by homogeneity $\sup_{\xi\neq 0}N(\xi)=\sup_{|\xi|=1}N(\xi)<\infty$. Indeed, if $\sup_{|\xi|=1}N(\xi)=+\infty$ we can find a sequence of points $(\xi_l)_l$ with $|\xi_l|=1$ and some $\xi'$ with $|\xi'|=1$ such that $\xi_l\to\xi'$ and $N(\xi_l)\to +\infty$ as $l\to\infty$. It follows that $f(t,\xi')$ must have infinite zeros in $t$ in contradiction with the hypothesis of analyticity on $[0,T]$.

We now work on the interval $(0,\tau_1)$. By the analiticity in $t$ we can write
\[
f(t,\xi)=t^{\nu_0(\xi)}(\tau_1-t)^{\nu_1(\xi)}g(t,\xi)
\]
where $g(t,\xi)$ is an analytic function in $t$ never vanishing on $[0,\tau_1]$ homogeneous of degree $0$ in $\xi$. Note that the functions $\nu_0$ and $\nu_1$ are positive and have
local maxima at all points (perturbations in $\xi$ in a sufficiently small neighborhood
can not increase the multiplicity). Arguing as in \cite[p.566]{KS} we write
$t|\partial_t f(t,\xi)|$ as
\[
\biggl|f(t)\biggl(\nu_0(\xi)-\frac{\nu_1(\xi)t}{(\tau_1-t)}+\frac{t\partial_t g(t,\xi)}{g(t,\xi)}\biggr)\biggr|
\]
Let us fix $\xi_0$ with $|\xi_0|=1$. Taking $t$ in $[0,\tau_1/2]$ and $\xi$ in a sufficiently small neighborhood of $\xi_0$ we have that $\nu_0(\xi)\le c_1$, $\nu_1(\xi)t/(\tau_1-t)
\le c_2$, $|g(t,\xi)|\ge c_0>0$ and ${t\partial_t g(t,\xi)}/{g(t,\xi)}\le c_3$. Hence,
\[
t|\partial_t f(t,\xi)|\le C |f(t,\xi)|
\]
on $[0,\tau_1/2]$ for $\xi$ in a neighborhood of $\xi_0$. Similarly, one proves that
\[
(\tau_1-t)|\partial_t f(t,\xi)|\le C |f(t,\xi)|
\]
on $[\tau_1/2,\tau_1]$ for $\xi$ in a neighborhood of $\xi_0$. The homogeneity in $\xi$ combined with a standard compactness argument allows us to extend the inequality
\[
|\partial_t f(t,\xi)|\le C\biggl(\frac{1}{t}+\frac{1}{\tau_{1}-t}\biggr)|f(t,\xi)|
\]
to $\R^n\setminus\{0\}$ for $t\in(0,\tau_1)$. Analogously one obtains that
\[
|\partial_t f(t,\xi)|\le C\biggl(\frac{1}{t-\tau_{h(\xi)}}+\frac{1}{\tau_{{h+1}(\xi)}-t}\biggr)|f(t,\xi)|
\]
when $t\in(\tau_{h(\xi)},\tau_{{h+1}(\xi)})$ and $\xi\neq 0$.
\end{proof}
In the case of analytic coefficients, Theorem \ref{theo_GR2i} follows from the
following Theorem \ref{theo_GR2}.
\begin{thm}
\label{theo_GR2}
If $A_j\in{C}([0,T])$ and the coefficients $A_{(j)}$ are analytic on $[0,T]$, the characteristic roots are real and satisfy \eqref{LC}, and the entries of the matrix $B$ of the lower order terms in \eqref{system_new} fulfill the conditions \eqref{LB} for $\xi$ away from $0$ then the Cauchy problem \eqref{CP} is $C^\infty$ and distributionally well-posed.
\end{thm}
\begin{proof}
By the finite propagation speed for hyperbolic equations
it is not restrictive to assume that the initial data $g_l$ are compactly supported. If the coefficients $a_j$ are analytic in $t$ on $[0,T]$ then by construction the entries of the quasi-symmetriser $Q_\eps^{(m)}$ are analytic as well. In particular, by Proposition \ref{prop_qs}
\[
q_{\eps,ij}(t,\xi)=q_{0,ij}(t,\xi)+\eps^2 q_{1,ij}(t,\xi)+\cdots+\eps^{2(m-1)}q_{m-1,ij}(t,\xi).
\]
We use the partition of the interval $[0,T]$ in Lemma \ref{lem_par} (applied to any $q_{\eps,ij}(t,\xi)$ or more precisely to any $\wt{q}_{\eps,ij}(t,\xi)=q_{\eps,ij}(t,\xi)\lara{\xi}/|\xi|$, homogeneous function of order $0$ in $\xi$ having the same zeros in $t$ of $q_{\eps,ij}(t,\xi)$). Considering the first interval $[0,\tau_1]$ ($\tau_1=\tau_1(\xi)$) we define
\[
E_\eps(t,\xi)=\begin{cases}
|V(t,\xi)|^2 & \text{for $t\in [0,\eps]\cup[\tau_1-\eps,\tau_1]$},\\
(Q_\eps(t,\xi)V(t,\xi),V(t,\xi)) & \text{for $t\in [\eps,\tau_1-\eps]$}.
\end{cases}
\]
as in \cite[p.567]{KS}. Hence
\begin{multline*}
\partial_t E_\eps(t,\xi)\le |\partial_t E_\eps(t,\xi)|\le |((A_1-A_1^\ast)V,V)|+|((B-B^\ast)V,V)|\\
\le \big(2\sup_{t\in[0,T]}\Vert A_1(t,\xi)\Vert +2\sup_{t\in[0,T]}\Vert B(t,\xi)\Vert\big)E_\eps(t,\xi)
\end{multline*}
on $[0,\eps]\cup[\tau_1-\eps,\tau_1]$. It follows by the Gronwall inequaity
that there exists a constant $\alpha>0$ such that
\beq
\label{estA}
E_\eps(t,\xi)\le
\begin{cases}
\esp^{2\alpha\eps\lara{\xi}}E_\eps(0,\xi) &\text{for $t\in[0,\eps]$},\\
\esp^{2\alpha\eps\lara{\xi}}E_\eps(\tau_1-\eps,\xi) &\text{for $t\in[\tau_1-\eps,\tau_1]$}.
\end{cases}
\eeq
On the interval $[\eps,\tau_1-\eps]$ we proceed as in the proof for the
Gevrey well-posedness under the conditions \eqref{LB} on the lower order terms for $|\xi|\ge R$. We have
\beq
\label{est_a}
\partial_t E_\eps(t,\xi)\le \biggl(\int_{\eps}^{\tau_1-\eps}\frac{|(\partial_tQ_\eps V,V)|}{(Q_\eps(t,\xi) V(t,\xi), V(t,\xi))}\, dt+C_2\eps\lara{\xi}+C_3\biggr)E_\eps(t,\xi).
\eeq
Since the family $Q_\eps(\lambda)$ is nearly diagonal when the roots $\lambda_j$ satisfy the condition \eqref{LC} we have that $Q_\eps\ge c_0\text{diag}\,Q_\eps$, i.e.,
\[
(Q_\eps(t,\xi)V,V)\ge c_0\sum_{h=1}^m q_{\eps,hh}(t,\xi)|V_h|^2.
\]
This fact combined with the inequality
\[
|q_{\eps,ij}||V_i||V_j|\le\sqrt{q_{\eps,ii}q_{\eps,jj}}|V_i||V_j|\le\sum_{h=1}^m q_{\eps,hh}|V_h|^2
\]
and Lemma \ref{lem_par} yields
\begin{multline*}
\int_{\eps}^{\tau_1-\eps}\frac{|(\partial_tQ_\eps V,V)|}{(Q_\eps(t,\xi) V(t,\xi), V(t,\xi))}\, dt\le c_0^{-1}\int_{\eps}^{\tau_1-\eps}\sum_{i,j=1}^m \frac{|\partial_t q_{\eps,ij}(t,\xi)|}{|q_{\eps,ij}(t,\xi)|}\, dt\\
\le C_1\int_\eps^{\tau_1-\eps}\biggl(\frac{1}{t}+\frac{1}{\tau_1-t}\biggr)\, dt =2C_1\log\frac{\tau_1-\eps}{\eps}\le 2C_1\log\frac{T}{\eps},
\end{multline*}
for some constant $C_1$ independent of $t$ and $\xi\neq 0$. Going back to estimate \eqref{est_a} we obtain
\[
\partial_t E_\eps(t,\xi)\le (2C_1\log\frac{T}{\eps}+C_2\eps\lara{\xi}+C_3)E_\eps(t,\xi)
\]
for $t\in[\eps,\tau_1-\eps]$ and $|\xi|\ge R$. This implies, by Gronwall's lemma, that
\beq
\label{estB}
E_\eps(t,\xi)\le C_TE_\eps(\eps,\xi)\esp^{C_T\log(1/\eps)+C_T\eps\lara{\xi}},
\eeq
on $[\eps,\tau_1-\eps]$. Finally, putting together \eqref{estA} with \eqref{estB} we conclude that there exists a constant $c>0$ such that
\[
E_\eps(t,\xi)\le cE_\eps(0,\xi)\esp^{c(\log(1/\eps)+\eps\lara{\xi})}
\]
for all $t\in[0,\tau_1]$ and $|\xi|\ge R$. Hence by applying \eqref{p2} we have
\[
|V(t,\xi)|\le c\eps^{-(m-1)}\esp^{C_T(\log(1/\eps)+\eps\lara{\xi})}|V(0,\xi)|
\]
on $[0,\tau_1]$. An iteration of the same technique on the other subintervals of $[0,T]$ leads to
\[
|V(t,\xi)|\le c\eps^{-N(\xi)(m-1)}\esp^{N(\xi)C_T(\log(1/\eps)+\eps\lara{\xi})}|V(0,\xi)|
\]
on $[0,T]$ for $|\xi|\ge R$. Now, setting $\eps=\lara{\xi}^{-1}$ we get
\[
|V(t,\xi)|\le c\lara{\xi}^{N(\xi)(m-1)}\esp^{N(\xi)C_T}\lara{\xi}^{N(\xi)C_T}.
\]
Remembering that from Lemma \ref{lem_par} the function $N(\xi)$ is bounded in $\xi$ we conclude that there exist some $\kappa\in\N$ and $C>0$ such that
\beq
\label{PW}
|V(t,\xi)|\le C\lara{\xi}^\kappa|V(0,\xi)|
\eeq
on $[0,T]$ for all $|\xi|\ge R$. It is clear that the estimate \eqref{PW} implies $C^\infty$ and distributional well-posedness of the Cauchy problem \eqref{CP}.
\end{proof}
Finally, given the energy estimates established above the proof of
Theorem \ref{theo_GR2ia} is simple:
\begin{proof}[Proof of Theorem \ref{theo_GR2ia}.]
We observe that the estimates \eqref{estv1} and \eqref{PW} imply that
$V(t,\xi)$ is bounded in $\xi$ if the lower order terms
$A(\cdot,\xi)$ are bounded on $[0,T]$. Coming back to the solution $u$
of \eqref{CP} and the definition of $V$ we get that
the solution $u(t,x)$ is in the class ${C}^{m-1}([0,T])$ with respect to $t$.
Finally, from the equality
$D^m_t u=-\sum_{j=0}^{m-1} A_{m-j}(t,D_x)D_t^j u$ we see that the
right hand side is bounded in $t$, implying that
$u(t,x)$ is in $W^{\infty,m}([0,T])$ with respect to $t$.
\end{proof}
We conclude the paper with the following remark on how the results change
if we assume less than the Levi conditions \eqref{EQ:lot}.
We thank T. Kinoshita for drawing our attention to this question.
\begin{remark}
\label{rem_final}
Note that the matrix $B$ of the lower order terms in \eqref{system_new}
can be written as
\[
B(t,\xi)=\sum_{l=0}^{m-1}B_{-l}(t,\xi),
\]
with
\[
B_{-l}=\left(
    \begin{array}{ccccc}
      0 & 0 & 0 & \dots & 0\\
      0 & 0 & 0& \dots & 0 \\
      \dots & \dots & \dots & \dots & 0 \\
      B_{-l,1} & B_{-l,2}& \dots & \dots & B_{-l,m}\\
    \end{array}
  \right)
\]
and
\[
B_{-l,j}(t,\xi)=\begin{cases}
-\sum_{|\gamma|=m-j-l} a_{m-j+1,\gamma}(t)\xi^\gamma\lara{\xi}^{j-m}, & \text{ for $j\le m-l$}\\
0, & \text{otherwise}
\end{cases}
\]
for $j=1,...,m$. We easily see that the matrix $B_{-l}$ has entries of order $-l$ and the last $l$ entries in the bottom row are equal to $0$. Making use of this decomposition of $B$ we can write $((Q_0^{(m)}B-B^\ast Q_0^{(m)})V,V)$ as
\beq
\label{decomp}
\sum_{l=0}^{m-1}((Q_0^{(m)}B_{-l}-B_{-l}^\ast Q_0^{(m)})V,V).
\eeq
\begin{trivlist}
\item[(i)] Let $0\le h\le m-2$. Let us assume the Levi conditions \eqref{EQ:lot} in the form \eqref{LB} only on the $B_{-l}$-matrices up to level $h$, i.e.,
instead of \eqref{EQ:lot} assume only that
\beq
\label{relaxLC}
\biggl|\sum_{l=0}^h B_{-l,j}\biggr|^2\prec \sum_{i=1}^m |\sigma_{m-j}^{(m-1)}(\pi_i\lambda)|^2,
\eeq
for $j=1,...,m$. In other words, we impose Levi conditions only on the coefficients of the equation corresponding to the matrices $B_{-l}$ up to $l=h$, leaving free the remaining lower order coefficients. Under these assumptions and the bound \eqref{p2} from below  for the quasi-symmetriser we obtain for \eqref{decomp} the estimate
\begin{multline*}
\biggl|\sum_{l=0}^{m-1}((Q_0^{(m)}B_{-l}-B_{-l}^\ast Q_0^{(m)})V,V)\biggr|\\
\le\biggl|\sum_{l=0}^{h}((Q_0^{(m)}B_{-l}-B_{-l}^\ast Q_0^{(m)})V,V)\biggr|+\biggl|\sum_{l=h+1}^{m-1}((Q_0^{(m)}B_{-l}-B_{-l}^\ast Q_0^{(m)})V,V)\biggr|\\
\le C_3E_\eps +C_4\lara{\xi}^{-h-1}\eps^{-2(m-1)}E_\eps.
\end{multline*}
This leads to the energy estimate
\[
\partial_t E_\eps(t,\xi)\le (C_1\eps^{-2(m-1)/k}+C_2\eps\lara{\xi}+C_3+C_4\lara{\xi}^{-h-1}\eps^{-2(m-1)})E_\eps(t,\xi).
\]
The Gevrey well-posedness result of Theorem \ref{theo_GRi} will still hold true under the relaxed Levi condition \eqref{relaxLC}, e.g.,
if for $\eps^{-2(m-1)/k}=\eps\lara{\xi}$ one has
\[
\lara{\xi}^{-h-1}\eps^{-2(m-1)}\le \lara{\xi}^{\frac{1}{\sigma}},
\]
with $\sigma=1+k/[2(m-1)]$, that is if
\beq
\label{h}
h+1\ge\frac{2(m-1)(k-1)}{k+2(m-1)}.
\eeq
In other words, for any fixed $k\in\N$ by involving sufficiently enough matrices $B_{-h}$ in the Levi condition \eqref{relaxLC} (how many depend on the equation order $m$ and the regularity $k$ of the coefficients) one can still obtain $G^s$ well-posedness for
\[
1\le s<1+\frac{k}{2(m-1)},
\]
but not the well-posedness in any Gevrey space even if $k$ increases to infinity. This is due to the fact that condition \eqref{h} implies $\frac{2(m-1)(k-1)}{k+2(m-1)}\le m-1$ and, therefore, gives the restriction $k\le 2m$.
\item[(ii)] Assume now that the equation coefficients are smooth. This implies that for any $a>0$ we can take $k$ large enough such that $\eps^{-2(m-1)/k}\le\lara{\xi}^a$. Hence, $\eps^{-2(m-1)/k}\le\lara{\xi}^{-h-1}\eps^{-2(m-1)}$ with $h=0$. Setting then $\eps\lara{\xi}=\lara{\xi}^{-1}\eps^{-2(m-1)}$ we get that under the Levi condition \eqref{relaxLC} with $h=0$ the Cauchy problem \eqref{CP} is well-posed in $G^s$ with
\[
1\le s<1+\frac{2}{2m-3}.
\]
In terms of Gevrey order this result is worse than the one stated in Theorem \ref{theo_GRi} but it is obtained with Levi conditions only on the coefficients appearing in the matrix $B_0$.
We note that it is better than the Bronstein's result due to the extra assumption
\eqref{LC} and the Levi condition \eqref{relaxLC} with $h=0.$
\end{trivlist}
\end{remark}

\bibliographystyle{abbrv}
\newcommand{\SortNoop}[1]{}

\end{document}